\documentclass[12pt]{amsart}
\usepackage{amsfonts, amssymb, mathrsfs, amsmath}
\usepackage{fullpage, verbatim, bbm}
\usepackage[colorlinks=true]{hyperref}
\usepackage{breakurl}

\usepackage{etoolbox}
\makeatletter
\let\ams@starttoc\@starttoc
\makeatother
\usepackage[parfill]{parskip}
\makeatletter
\let\@starttoc\ams@starttoc
\patchcmd{\@starttoc}{\makeatletter}{\makeatletter\parskip\z@}{}{}
\makeatother

\newcommand\CA{{\mathscr A}}

\newcommand\CIF{{\mathcal {IF}}} 
 
\newcommand\CIFM{{\mathcal {IFM}}}

\newcommand\BBC{{\mathbb C}}
\newcommand\BBK{{\mathbb K}}

\newcommand\BBZ{{\mathbb Z}}

\newcommand\Der{{\operatorname{Der}}}

\newcommand\GL{\operatorname{GL}}

\newcommand\pdeg{\operatorname{pdeg}}

\newcommand\rank{\operatorname{rank}}

\renewcommand\th{{^{\text{th}}}}

\newcommand{\one}{\mathbbm{1}}

\numberwithin{equation}{section}

\theoremstyle{plain}
\newtheorem{lemma}[equation]{Lemma}
\newtheorem{theorem}[equation]{Theorem}

\newtheorem{corollary}[equation]{Corollary}
\newtheorem{proposition}[equation]{Proposition}
\theoremstyle{definition}
\newtheorem{defn}[equation]{Definition}
\newtheorem{remark}[equation]{Remark}

\subjclass[2010]{52C35 (14N20, 32S22, 51D20)}  
\begin{document}

%%%%%%%%%%%%%%%%%%%%%%%%%%%%%%%%%%%%%%%%%%%%%%%%%%%%%%%%%%%%%%%%%%%%%%
%%%%%%%%%%%%% top matter stuff
%%%%%%%%%%%%%%%%%%%%%%%%%%%%%%%%%%%%%%%%%%%%%%%%%%%%%%%%%%%%%%%%%%%%%%
\title[Inductive Freeness of Ziegler's Canonial Multiderivations]
{Inductive Freeness of Ziegler's Canonical Multiderivations for Restrictions of Reflection Arrangements}

\author[T.~Hoge, G.~R\"ohrle, and S.~Wiesner]{Torsten Hoge, Gerhard R\"ohrle, and Sven Wiesner}
\address
{Fakult\"at f\"ur Mathematik,
	Ruhr-Universit\"at Bochum,
	D-44780 Bochum, Germany}
\email{torsten.hoge@rub.de}
\email{gerhard.roehrle@rub.de}
\email{sven.wiesner@rub.de}

\keywords{
Free arrangement, free
multiarrangement,
Ziegler multiplicity,
inductively free arrangement,
reflection arrangement, restrictions of reflection arrangements}

\allowdisplaybreaks

\begin{abstract}
Let $\CA$ be a free hyperplane arrangement.  
In 1989, Ziegler showed that 
the restriction $\CA''$ of $\CA$
to any hyperplane 
endowed with the natural multiplicity $\kappa$ 
is then a free multiarrangement.
Recently, in \cite{hogeroehrle:ZieglerII},  
an analogue of Ziegler's theorem for the stronger notion of inductive freeness was proved:  
if $\CA$ is inductively free, then so is 
the free multiarrangement $(\CA'',\kappa)$.

In \cite{hogeroehrle:Ziegler}, all reflection arrangements which admit inductively free Ziegler restrictions are classified. The aim of this paper is to extend this classification to all arrangements which are induced by reflection arrangements
utilizing the aforementioned fundamental result from \cite{hogeroehrle:ZieglerII}.
\end{abstract}

\maketitle

%\setcounter{tocdepth}{2}
%\tableofcontents

%%%%%%%%%%%%%%%%%%%%%%%%%%%%%%%%%%%%%%%%%%%%%%%%%%%%%%%%%%%%%%%%%%%%%%
%%%%%%%%%%%%% article body...
%%%%%%%%%%%%%%%%%%%%%%%%%%%%%%%%%%%%%%%%%%%%%%%%%%%%%%%%%%%%%%%%%%%%%%

%%%%%%%%%%%%%%%%%%%%%%%%%%%%%%%%%%%%%%%%%%%%%%%%%%%%%%%%%%%%%%%%%%%%%%
%%%%%%%%%%%%% \S1 Introduction
%%%%%%%%%%%%%%%%%%%%%%%%%%%%%%%%%%%%%%%%%%%%%%%%%%%%%%%%%%%%%%%%%%%%%%
\section{Introduction}

\subsection{}
The class of free arrangements, respectively free multiarrangements, 
plays a fundamental role in the theory of 
hyperplane arrangements, respectively 
multiarrangements. 
In \cite{ziegler:multiarrangements}, Ziegler 
introduced the notion of multiarrangements and initiated the study of their 
freeness.  
We begin by recalling Ziegler's 
fundamental construction from \cite{ziegler:multiarrangements}. 

Let $\BBK$ be a field and let $V$ 
be a finite-dimensional $\BBK$-vector space.
A \emph{hyperplane arrangement} or simply an \emph{arrangement} is a pair
$(\CA, V)$, where $\CA$ is a finite collection of (central) hyperplanes in $V$.
Usually, we just write $\CA$ in place of $(\CA, V)$; see Section \ref{ssect:hyper}.

\begin{defn}
\label{def:kappa}
Let $\CA$ be an arrangement.
Fix $H_0 \in \CA$ and  consider the restriction 
$\CA''$ with respect to $H_0$.
Define the \emph{canonical multiplicity} 
$\kappa$ on $\CA''$ as follows. For $Y \in \CA''$ set 
\[
\kappa(Y) := |\CA_Y| -1,
\]
i.e., $\kappa(Y)$ is the number of hyperplanes in $\CA \setminus\{H_0\}$
lying above $Y$.
Ziegler showed that freeness of $\CA$ implies 
freeness of the multiarrangement $(\CA'', \kappa)$ as follows.
\end{defn}

\begin{theorem}
[{\cite[Thm.~11]{ziegler:multiarrangements}}]
\label{thm:zieglermulti}
Let $\CA$ be a free arrangement with exponents
$\exp \CA = \{1, e_2, \ldots, e_\ell\}$.
Let $H_0 \in \CA$ and consider the restriction 
$\CA''$ with respect to $H_0$.
Then the multiarrangement $(\CA'', \kappa)$ is free with
exponents
$\exp (\CA'', \kappa) = \{e_2, \ldots, e_\ell\}$. 
\end{theorem}

Because of the relevance of Ziegler's multiplicity
in the theory of free arrangements, 
it is natural to investigate stronger
freeness properties for $(\CA'', \kappa)$ and
specifically to ask for an analogue of
Theorem \ref{thm:zieglermulti} for 
inductive freeness. This was settled recently in \cite{hogeroehrle:ZieglerII}:

\begin{theorem}[{\cite[Thm.~1.3]{hogeroehrle:ZieglerII}}]
	\label{thm:main}
	If $\CA$ is inductively free, then so is the Ziegler restriction $(\CA'',\kappa)$, for every $H_0 \in \CA$.
\end{theorem}

Even if $\CA$ itself is not inductively free, the following result from \cite{hogeroehrle:ZieglerII} still guarantees 
inductive freeness  of $(\CA'',\kappa)$, provided  the restriction $\CA''$ itself is inductively free.

\begin{theorem}[{\cite[Thm.~1.6]{hogeroehrle:ZieglerII}}]
	\label{thm:main2}
	Let $H_0 \in \CA$. Suppose that $\CA\setminus \{H_0\}$ is free.
	If $\CA''$ is inductively free, then so is $(\CA'',\kappa)$.
\end{theorem}

We note that here the freeness assumption on $\CA\setminus \{H_0\}$ is needed, cf.~\cite[Ex.~4.2]{hogeroehrle:ZieglerII}.

Because of the compatibility of
the product constructions for  
inductive freeness for simple arrangements,
\cite[Prop.~2.10]{hogeroehrle:indfree},
as well as for multiarrangements,
\cite[Thm.~1.4]{hogeroehrleschauenburg:free}
(cf.~Theorem \ref{thm:products}),
the question of inductive freeness 
of $(\CA'', \kappa)$ readily reduces to 
the case when $\CA$ is irreducible.

\subsection{} 
In this subsection assume that   $\BBK = \BBC$,  the complex numbers, and let $W$ be an irreducible complex reflection group with $V$ the $\BBC$-vector space affording the reflection representation of $W$. Denote by $\CA(W)$ the 
\emph{reflection arrangement} of $W$ in $V$ which consists of all the complex reflecting hyperplanes of $W$.
In \cite{hogeroehrle:Ziegler},  all reflection arrangements $\CA(W)$ were classified
which admit Ziegler restrictions $(\CA(W)'', \kappa)$ that 
are inductively free:

\begin{theorem}
	[{\cite[Thm.~1.4]{hogeroehrle:Ziegler}}]
	\label{thm:kappa1}
	Let $\CA = \CA(W)$ be the reflection arrangement of 
	the irreducible complex reflection group $W$.
	Let $\CA''$ be the restriction of $\CA$ to 
	a hyperplane in $\CA$.
	Then $(\CA'', \kappa)$ is inductively free
	if and only if one of the following holds:
	\begin{itemize}
		\item[(i)] $\CA$ is inductively free; or 
		\item[(ii)] $\CA$ is non-inductively free of rank at most $4$. 
	\end{itemize}
\end{theorem}

We record a 
consequence of Theorem \ref{thm:kappa1}
along with  the classification of all 
inductively free restrictions of 
reflection arrangements from 
\cite[Thm.~1.2]{amendhogeroehrle:indfree},
(cf.~Theorem \ref{thm:indfree2}).

\begin{corollary}
	[{\cite[Cor.~1.5]{hogeroehrle:Ziegler}}]
	\label{cor:kappa1}
	Let $\CA = \CA(W)$ be the reflection arrangement of 
	the complex reflection group $W$. 	Let $\CA''$ be the restriction of $\CA$ to 
	a hyperplane in $\CA$.
	Then $(\CA'', \kappa)$ is inductively free
	if and only if $\CA''$ itself is inductively free.
\end{corollary}

The goal of this paper is to extend the classification of 
Theorem  \ref{thm:kappa1} and Corollary \ref{cor:kappa1} to restrictions $\CA(W)^X$ of complex reflection arrangements $\CA(W)$.
Clearly, the question reduces 
readily to the case when $W$ is irreducible.
Thanks to work by 
Orlik--Terao \cite{orlikterao:free}, 
\cite[Prop.\ 6.73, Prop.\ 6.77, Cor.\ 6.86, App.\ D]{orlikterao:arrangements}
and Hoge--R\"ohrle \cite{hogeroehrle:free},
every restriction 
$\CA(W)^X$ of a reflection arrangement $\CA(W)$ 
is free for $W$ an irreducible complex reflection group.
So in particular, the  
Ziegler multiplicity of a restriction of $\CA(W)^X$ to 
a hyperplane is free, by Theorem~\ref{thm:zieglermulti}.
Here is the principal result of our paper.

\begin{theorem}
	\label{thm:kappa2}
	Let $W$ be a finite, irreducible, complex 
	reflection group with reflection arrangement 
	$\CA(W)$. Fix $X \in L(\CA(W))\setminus\{V\}$ and  
	let $\CA = \CA(W)^X$ be the restriction of 
	$\CA(W)$ to $X$.
	Let $\CA''$ be the restriction of $\CA$ to 
	a hyperplane in $\CA$.
	Then $(\CA'', \kappa)$ is inductively free
	if and only if one of the following holds:
	\begin{itemize}
		\item[(i)] $\CA(W)$ is inductively free; or
		\item[(ii)] $W = G(r,r,\ell)$ and the restriction $\CA''$ of $\CA$ 
		itself is inductively free; or
		\item[(iii)] $W$ is of exceptional type $G_{33}$ or $G_{34}$ and $\rank \CA \le 4$.	
	\end{itemize}
\end{theorem}

Theorem  \ref{thm:kappa1} shows that the question of inductive freeness of $(\CA'', \kappa)$ for $\CA = \CA(W)$ a complex reflection arrangement does not depend on the choice of hyperplane. In contrast for restrictions of complex reflection arrangements this question is sensitive to the choice of hyperplane for the restriction, 
according to Theorem \ref{thm:kappa2}(ii) in view of Theorem \ref{thm:indfree2}(i)(b).

The next observation follows from Theorem \ref{thm:kappa2} along with  
the classification of all 
inductively free restrictions of 
reflection arrangements from 
\cite[Thm.~1.2]{amendhogeroehrle:indfree},
see Theorem \ref{thm:indfree2}, which is the analogue of Corollary \ref{cor:kappa1} for the class of restrictions of 
reflection arrangements.

\begin{corollary}
\label{cor:kappa2}
	Let $W$ be a finite, irreducible, complex 
reflection group with reflection arrangement 
$\CA(W)$ and let $X \in L(\CA(W))$. 
Let $\CA = \CA(W)^X$ be the restriction of 
$\CA(W)$ to $X$.
	Let $\CA''$ be the restriction of $\CA$ to 
	a hyperplane in $\CA$.
Then $(\CA'', \kappa)$ is inductively free
if and only if $\CA''$ itself is inductively free. 
\end{corollary}

We mention that the equivalences from both Corollaries \ref{cor:kappa1} and \ref{cor:kappa2}
are false in general
 (i.e., if $\CA$ is not a reflection arrangement, and not a restriction of such).
For the failure of the forward implication, see
\cite[Ex.~1.13]{hogeroehrle:ZieglerII} and for the failure of 
the reverse implication, see
\cite[Ex.~2.24]{hogeroehrle:Ziegler}.
In view of these elementary counterexamples, 
the equivalences of these two corollaries 
are rather striking.

\subsection{} 
In this subsection we allow $\BBK$ to be an  arbitrary field again.
In a separate part, we also study 
multiplicities which are
concentrated at a single hyperplane. 
These were introduced by
Abe, Terao and Wakefield,
{\cite[\S 5]{abeteraowakefield:euler}.
	It turns out that they are closely related to Ziegler's
	canonical multiplicity, 
	see Proposition \ref{prop:delta}. 
	
	\begin{defn}
		\label{def:delta}
		Let $\CA$ be a simple arrangement.
		Fix $H_0 \in \CA$ and $m_0 \in \BBZ_{\ge 1}$ and 
		define the 
		\emph{multiplicity $\delta$ concentrated at $H_0$}
		by
		\[
		\delta(H) := \delta_{H_0,m_0}(H) := 
		\begin{cases}
		m_0 & \text{ if } H = H_0,\\
		1   & \text{ else}.
		\end{cases}
		\]
	\end{defn}
	In general a multiarrangement  $(\CA, \mu)$ need not be free 
for a free hyperplane arrangement $\CA$ and 
an arbitrary multiplicity $\mu$, 
e.g.~see \cite[Ex.~14]{ziegler:multiarrangements}.
However, as opposed to the general case, both $\CA$ and $(\CA, \delta)$ do 
	inherit freeness from each other.
	
	\begin{theorem}
			[{\cite[Thm.~1.7]{hogeroehrle:Ziegler}}]
		\label{thm:delta-free}
		Let $\CA$ be an arrangement.
		Fix $H_0 \in \CA$, $m_0 \in \BBZ_{\ge 1}$ and let 
		$\delta = \delta_{H_0,m_0}$ be 
		the multiplicity concentrated at $H_0$.
		Then $\CA$ is free 
		with exponents
		$\exp \CA = \{1, e_2, \ldots, e_\ell\}$
		if and only if
		$(\CA, \delta)$ is free with exponents
		$\exp (\CA, \delta) = \{m_0, e_2, \ldots, e_\ell\}$. 
	\end{theorem}
	
Our aim here is to strengthen Theorem \ref{thm:delta-free} to inductive freeness.

	\begin{theorem}
	\label{thm:delta-indfree}
	Let	$H_0 \in \CA$, $m_0 \in \BBZ_{\ge 1}$ and let 
	$\delta = \delta_{H_0,m_0}$ be 
	the multiplicity concentrated at $H_0$.
	Then $\CA$ is inductively free 
	if and only if
	$(\CA, \delta)$ is inductively free.
\end{theorem}

The forward implication of Theorem \ref{thm:delta-indfree}
is \cite[Cor.~5.6]{hogeroehrle:ZieglerII}.
We observe that Theorem \ref{thm:delta-indfree} gives a uniform case free proof of 
\cite[Thm.~1.9]{hogeroehrle:Ziegler}.

The following combines
\cite[Prop.~5.2]{abeteraowakefield:euler}, parts of its proof
and Theorem \ref{thm:zieglermulti}. While not needed for the proof of 
Theorem \ref{thm:delta-indfree}, this result does illustrate the connection with the canonical multiplicity.

\begin{proposition}[{\cite[Prop.~2.14]{hogeroehrle:Ziegler}}]
	\label{prop:delta}
	Let $\CA$ be a free arrangement with exponents 
	$\exp \CA = \{1, e_2, \ldots, e_\ell\}$.
	Let $H_0 \in \CA$, $m_0 \ge 1$, and  
	$\delta = \delta_{H_0,m_0}$. 
	Let  $(\CA'', \delta^*)$ be the restriction with respect to $H_0$. 
	Then we have
	\begin{itemize}
		\item[(i)] 
		$(\CA, \delta)$ is free with exponents
		$\exp (\CA, \delta) = \{m_0, e_2, \ldots, e_\ell\}$;
		\item[(ii)] 
		$(\CA'', \delta^*) = (\CA'', \kappa)$ is free with exponents
		$\exp  (\CA'', \kappa) = \{e_2, \ldots, e_\ell\}$.
	\end{itemize}
\end{proposition}

The paper is organized as follows.
In Section \ref{ssect:hyper} we recall 
the fundamental results for free arrangements, 
in particular Terao's 
Addition Deletion Theorem \ref{thm:add-del-simple}
and the subsequent
notion of an inductively free arrangement.

Section \ref{ssec:multi} is devoted to 
multiarrangements and their freeness. 
Here we present the Addition Deletion Theorem 
due to Abe, Terao, and Wakefield, 
\cite[Thm.~0.8]{abeteraowakefield:euler}, see Theorem \ref{thm:add-del}.
This is followed by a 
discussion of inductive freeness 
for multiarrangements. 
Here we recall results from 
\cite{hogeroehrleschauenburg:free}
which show the compatibility of this notion with 
products and localization for multiarrangements
that are used in the sequel.

In Section \ref{ssect:refl}
we recall the classification of 
the inductively free reflection arrangements from 
\cite[Thm.~1.1, 1.2]{hogeroehrle:indfree} (Theorem \ref{thm:indfree1})
and the classification of 
the inductively free restrictions of 
reflection arrangements \cite[Thm.~1.2, Thm.~1.3]{amendhogeroehrle:indfree} (Theorem \ref{thm:indfree2}).
We combine the latter with 
Theorems \ref{thm:main} and \ref{thm:main2} to obtain a first result towards Theorem \ref{thm:kappa2} in 
Theorem \ref{thm:indfreemain1}.

Theorem \ref{thm:kappa2}  is proved in Section
\ref{sec:proofs}
and 
Theorem \ref{thm:delta-indfree} in Section
\ref{sec:proof}.

\section{Recollections and Preliminaries}
\label{sect:prelim}

\subsection{Freeness of hyperplane arrangements}
\label{ssect:hyper}
Let $V = \BBK^\ell$ 
be an $\ell$-dimensional $\BBK$-vector space.
A \emph{hyperplane arrangement} is a pair
$(\CA, V)$, where $\CA$ is a finite collection of hyperplanes in $V$.
Usually, we simply write $\CA$ in place of $(\CA, V)$.
We write $|\CA|$ for the number of hyperplanes in $\CA$.
The empty arrangement in $V$ is denoted by $\Phi_\ell$.

The \emph{lattice} $L(\CA)$ of $\CA$ is the set of subspaces of $V$ of
the form $H_1\cap \dotsm \cap H_i$ where $\{ H_1, \ldots, H_i\}$ is a subset
of $\CA$. 
For $X \in L(\CA)$, we have two associated arrangements, 
firstly
$\CA_X :=\{H \in \CA \mid X \subseteq H\} \subseteq \CA$,
the \emph{localization of $\CA$ at $X$}, 
and secondly, 
the \emph{restriction of $\CA$ to $X$}, $(\CA^X,X)$, where 
$\CA^X := \{ X \cap H \mid H \in \CA \setminus \CA_X\}$.
Note that $V$ belongs to $L(\CA)$
as the intersection of the empty 
collection of hyperplanes and $\CA^V = \CA$. 
The lattice $L(\CA)$ is a partially ordered set by reverse inclusion:
$X \le Y$ provided $Y \subseteq X$ for $X,Y \in L(\CA)$.

Suppose $\CA \neq \Phi_\ell$. Fix a member $H_0$ in $\CA$. Set $\CA' = \CA \setminus \{H_0\}$ and  $\CA'' = \CA^{H_0}$. Then $(\CA, \CA', \CA'')$ is frequently referred to as a \emph{triple of arrangements}. 

Let $S = S(V^*)$ be the symmetric algebra of the dual space $V^*$ of $V$.
If $x_1, \ldots , x_\ell$ is a basis of $V^*$, then we identify $S$ with 
the polynomial ring $\BBK[x_1, \ldots , x_\ell]$.
Letting $S_p$ denote the $\BBK$-subspace of $S$
consisting of the homogeneous polynomials of degree $p$ (along with $0$),
$S$ is naturally $\BBZ$-graded: $S = \oplus_{p \in \BBZ}S_p$, where
$S_p = 0$ in case $p < 0$.

Let $\Der(S)$ be the $S$-module of algebraic $\BBK$-derivations of $S$.
Using the $\BBZ$-grading on $S$, $\Der(S)$ becomes a graded $S$-module.
For $i = 1, \ldots, \ell$, 
let $D_i := \partial/\partial x_i$ be the usual derivation of $S$.
Then $D_1, \ldots, D_\ell$ is an $S$-basis of $\Der(S)$.
We say that $\theta \in \Der(S)$ is 
\emph{homogeneous of polynomial degree p}
provided 
$\theta = \sum_{i=1}^\ell f_i D_i$, 
where $f_i$ is either $0$ or homogeneous of degree $p$
for each $1 \le i \le \ell$.
In this case we write $\pdeg \theta = p$.

Let $\CA$ be an arrangement in $V$. 
Then for $H \in \CA$ we fix $\alpha_H \in V^*$ with
$H = \ker(\alpha_H)$.
The \emph{defining polynomial} $Q(\CA)$ of $\CA$ is given by 
$Q(\CA) := \prod_{H \in \CA} \alpha_H \in S$.

The \emph{module of $\CA$-derivations} of $\CA$ is 
defined by 
\[
D(\CA) := \{\theta \in \Der(S) \mid \theta(\alpha_H) \in \alpha_H S
\text{ for each } H \in \CA \} .
\]
We say that $\CA$ is \emph{free} if the module of $\CA$-derivations
$D(\CA)$ is a free $S$-module.

With the $\BBZ$-grading of $\Der(S)$, 
also $D(\CA)$ 
becomes a graded $S$-module,
\cite[Prop.~4.10]{orlikterao:arrangements}.
If $\CA$ is a free arrangement, then the $S$-module 
$D(\CA)$ admits a basis of $\ell$ homogeneous derivations, 
say $\theta_1, \ldots, \theta_\ell$, \cite[Prop.~4.18]{orlikterao:arrangements}.
While the $\theta_i$'s are not unique, their polynomial 
degrees $\pdeg \theta_i$ 
are unique (up to ordering). This multiset is the set of 
\emph{exponents} of the free arrangement $\CA$
and is denoted by $\exp \CA$.

The fundamental \emph{Addition Deletion Theorem} 
due to Terao  \cite{terao:freeI} plays a 
crucial role in the study of free arrangements, 
\cite[Thm.~4.51]{orlikterao:arrangements}.

\begin{theorem}
\label{thm:add-del-simple}
Suppose $\CA \neq \Phi_\ell$ and
let $(\CA, \CA', \CA'')$ be a triple of arrangements. Then any 
two of the following statements imply the third:
\begin{itemize}
\item[(i)] $\CA$ is free with $\exp\CA = \{ b_1, \ldots , b_{\ell -1}, b_\ell\}$;
\item[(ii)] $\CA'$ is free with $\exp\CA' = \{ b_1, \ldots , b_{\ell -1}, b_\ell-1\}$;
\item[(iii)] $\CA''$ is free with $\exp\CA'' = \{ b_1, \ldots , b_{\ell -1}\}$.
\end{itemize}
\end{theorem}

Theorem \ref{thm:add-del-simple} motivates the following notion.

\begin{defn}
	[{\cite[Def.~4.53]{orlikterao:arrangements}}]
\label{def:indfree-simple}
The class $\CIF$ of \emph{inductively free} arrangements 
is the smallest class of arrangements subject to
\begin{itemize}
\item[(i)] $\Phi_\ell \in \CIF$ for each $\ell \ge 0$;
\item[(ii)] if there exists a hyperplane $H_0 \in \CA$ such that both
$\CA'$ and $\CA''$ belong to $\CIF$, and $\exp \CA '' \subseteq \exp \CA'$, 
then $\CA$ also belongs to $\CIF$.
\end{itemize}
\end{defn}

\begin{defn}
	[{\cite[p.~253]{orlikterao:arrangements}}]
	\label{def:heredfree-simple}
	An arrangement $\CA$ is called \emph{hereditarily inductively free} if $\CA$ and every restriction of $\CA$ is inductively free.
\end{defn}

\subsection{Freeness of  multiarrangements}
\label{ssec:multi}
A \emph{multiarrangement}  is a pair
$(\CA, \mu)$ consisting of a hyperplane arrangement $\CA$ and a 
\emph{multiplicity} function
$\mu : \CA \to \BBZ_{\ge 0}$ associating 
to each hyperplane $H$ in $\CA$ a non-negative integer $\mu(H)$.
The \emph{order} of the multiarrangement $(\CA, \mu)$ 
is defined by 
$|\mu| := 
%|(\CA, \mu)| = 
\sum_{H \in \CA} \mu(H)$.
For a multiarrangement $(\CA, \mu)$, the underlying 
arrangement $\CA$ is sometimes called the associated 
\emph{simple} arrangement, and so $(\CA, \mu)$ itself is  
simple if and only if $\mu(H) = 1$ for each $H \in \CA$. We indicate this case with the notation $\mu = \one$. 
In case $\mu \equiv 0$, $(\CA, \mu)$ is the empty arrangement.

Let $(\CA, \mu)$ be a multiarrangement in $V$ and let 
$X \in L(\CA)$. The 
\emph{localization of $(\CA, \mu)$ at $X$} is defined to be $(\CA_X, \mu_X)$,
where $\mu_X = \mu |_{\CA_X}$.

The \emph{defining polynomial} $Q(\CA, \mu)$ 
of the multiarrangement $(\CA, \mu)$ is given by 
\[
Q(\CA, \mu) := \prod_{H \in \CA} \alpha_H^{\mu(H)},
\] 
a polynomial of degree $|\mu|$ in $S$.

Following Ziegler \cite{ziegler:multiarrangements},
we extend the notion of freeness to multiarrangements as follows.
The \emph{module of $\CA$-derivations} of $(\CA, \mu)$ is 
defined by 
\[
D(\CA, \mu) := \{\theta \in \Der(S) \mid \theta(\alpha_H) \in \alpha_H^{\mu(H)} S 
\text{ for each } H \in \CA\}.
\]
We say that $(\CA, \mu)$ is \emph{free} if 
$D(\CA, \mu)$ is a free $S$-module, 
\cite[Def.~6]{ziegler:multiarrangements}.

As in the case of simple arrangements,
$D(\CA, \mu)$ is a $\BBZ$-graded $S$-module and 
thus, if $(\CA, \mu)$ is free, there is a 
homogeneous basis $\theta_1, \ldots, \theta_\ell$ of $D(\CA, \mu)$.
The multiset of the unique polynomial degrees $\pdeg \theta_i$ 
forms the set of \emph{exponents} of the free multiarrangement $(\CA, \mu)$
and is denoted by $\exp (\CA, \mu)$.

We recall the construction from \cite{abeteraowakefield:euler} for the 
counterpart of Theorem \ref{thm:add-del-simple} in this more general setting.

\begin{defn}
\label{def:Euler}
Let $(\CA, \mu) \ne \Phi_\ell$ be a multiarrangement. Fix $H_0$ in $\CA$.
We define the \emph{deletion}  $(\CA', \mu')$ and \emph{Euler restriction} $(\CA'', \mu^*)$
of $(\CA, \mu)$ with respect to $H_0$ as follows.
If $\mu(H_0) = 1$, then set $\CA' = \CA \setminus \{H_0\}$
and define $\mu'(H) = \mu(H)$ for all $H \in \CA'$.
If $\mu(H_0) > 1$, then set $\CA' = \CA$
and define $\mu'(H_0) = \mu(H_0)-1$ and
$\mu'(H) = \mu(H)$ for all $H \ne H_0$.

Let $\CA'' = \{ H \cap H_0 \mid H \in \CA \setminus \{H_0\}\ \}$.
The \emph{Euler multiplicity} $\mu^*$ of $\CA''$ is defined as follows.
Let $Y \in \CA''$. Since the localization $\CA_Y$ is of rank $2$, the
multiarrangement $(\CA_Y, \mu_Y)$ is free, 
\cite[Cor.~7]{ziegler:multiarrangements}. 
According to 
\cite[Prop.~2.1]{abeteraowakefield:euler},
the module of derivations 
$D(\CA_Y, \mu_Y)$ admits a particular homogeneous basis
$\{\theta_Y, \psi_Y, D_3, \ldots, D_\ell\}$,
such that $\theta_Y \notin \alpha_0 \Der(S)$
and $\psi_Y \in \alpha_0 \Der(S)$,
where $H_0 = \ker \alpha_0$.
Then on $Y$ the Euler multiplicity $\mu^*$ is defined
to be $\mu^*(Y) = \pdeg \theta_Y$.

Often, we refer to 
$(\CA, \mu), (\CA', \mu')$ and $(\CA'', \mu^*)$ 
as the \emph{triple} of 
$(\CA, \mu)$ with respect to $H_0$. 
\end{defn}

The following observation is a direct consequence of Definition \ref{def:Euler}.

\begin{remark}
	\label{rem:euler}
	Let $(\CA,\mu)$ be a multiarrangement, $H_0\in \CA$, and let $(\CA'',\mu^*)$ be the Euler restriction corresponding to $H_0$. For $X\in\CA''$ the value $\mu^*(X)$ is simply the common value of the non-zero exponents of $(\CA_X,\mu_X)$ and $(\CA_X',\mu_X')$.	
\end{remark} 

It is useful to be able to determine the Euler multiplicity explicitly in some 
relevant instances. For that purpose, 
we recall parts of \cite[Prop.~4.1]{abeteraowakefield:euler}.

\begin{proposition}
	\label{ATWEulerProp}
	Let $(\CA,\mu)$ be a multiarrangement. Let $X\in\CA^{H_0}$, $m_0=\mu(H_0)$, $k=\vert \CA_X\vert$, and $m_1=\max\{\mu(H)\mid H\in \CA_X\backslash\{H_0\}\}$.
	\begin{enumerate}
		\item If $k=2$, then $\mu^*(X)=m_1$.
		\item If $\vert \mu_X\vert\leq 2k-1$ and $m_0>1$, then $\mu^*(X)=k-1$.
		\item If $\mu_X\equiv 2$, then $\mu^*(X)=k$.
	\end{enumerate}
\end{proposition}

The following is a special case of a general 
compatibility result 
of Euler restrictions with localizations from 
\cite[Lem.~2.14]{hogeroehrleschauenburg:free}.

\begin{lemma}
	[{\cite[Lem.~2.16]{hogeroehrle:Ziegler}}]
	\label{lem:euler}
	Let $X \in L(\CA)$, $H_0 \in \CA_X$,
	and let $\CA''$ be the restriction 
	with respect to $H_0$. Let $\delta = \delta_{H_0,m_0}$ 
	be as in Definition \ref{def:delta}. 
	Then we have 
	\begin{itemize}
		\item[(i)]
		$\left((\CA_X)'', (\delta_X)^*\right) = \left((\CA'')_X, (\delta^*)_X\right)$, and 
		\item[(ii)]
		$\left((\CA_X)'', \kappa\right) = \left((\CA'')_X, \kappa_X\right)$ 
		(where $\kappa$ on the left is the canonical multiplicity 
		resulting from restriction of $\CA_X$ to $H_0$).
	\end{itemize}
\end{lemma}

We also require the following tool from \cite{hogeroehrle:ZieglerII}.
This result turns out to be useful in our sequel in determining inductively free Ziegler restrictions. 

\begin{lemma}[{\cite[Lem.~2.12]{hogeroehrle:ZieglerII}}]
	\label{lemma:free_sequence_to_ziegler_multiplicity}
	Let $\CA$ be an $\ell$-arrangement and let $H_1$,\dots,$H_n$ be distinct hyperplanes in $\CA$.
	Define $\CA_0 := \CA \setminus \{H_1,\ldots,H_n\}$, and $\CA_i := \CA_{i-1} \cup \{H_i\}$ for $i = 1, \ldots, n$. 
	Suppose that $\CA_0,\ldots,\CA_n = \CA$ are all free.  Then for a fixed $H \in \CA_0$, we have the following:
	\begin{enumerate}
		\item 
		$(\CA_i^H,\kappa_i)$ is free for $i=0,\ldots,n$, where $\kappa_i$ is the Ziegler multiplicity of $\CA_i^H$.
		
		\item Consider the triple $(\CA_{i}^H,\kappa_i)$, $(\CA_{i-1}^H,\kappa_{i-1})$,  and $((\CA_i^H)^{H \cap H_i}, \kappa_i^*)$. Then the Euler multiplicity  $\kappa_i^*$ is just the 
		Ziegler multiplicity $\kappa$ of the restriction of $\CA_i^{H_i}$ to $H \cap H_i$. 
	\end{enumerate}
\end{lemma}

Here is the counterpart to Theorem \ref{thm:add-del-simple} in this more general setting from 
\cite{abeteraowakefield:euler}. 

\begin{theorem}
[{\cite[Thm.~0.8]{abeteraowakefield:euler}}
Addition Deletion Theorem for Multiarrangements]
\label{thm:add-del}
Suppose that $(\CA, \mu) \ne \Phi_\ell$.
Fix $H_0$ in $\CA$ and 
let  $(\CA, \mu), (\CA', \mu')$ and  $(\CA'', \mu^*)$ be the triple with respect to $H_0$. 
Then any  two of the following statements imply the third:
\begin{itemize}
\item[(i)] $(\CA, \mu)$ is free with $\exp (\CA, \mu) = \{ b_1, \ldots , b_{\ell -1}, b_\ell\}$;
\item[(ii)] $(\CA', \mu')$ is free with $\exp (\CA', \mu') = \{ b_1, \ldots , b_{\ell -1}, b_\ell-1\}$;
\item[(iii)] $(\CA'', \mu^*)$ is free with $\exp (\CA'', \mu^*) = \{ b_1, \ldots , b_{\ell -1}\}$.
\end{itemize}
\end{theorem}

Analogous to the simple case, Theorem \ref{thm:add-del} motivates 
the notion of inductive freeness. 

\begin{defn}[{\cite[Def.~0.9]{abeteraowakefield:euler}}]
\label{def:indfree}
The class $\CIFM$ of \emph{inductively free} multiarrangements 
is the smallest class of multiarrangements subject to
\begin{itemize}
\item[(i)] $\Phi_\ell \in \CIFM$ for each $\ell \ge 0$;
\item[(ii)] for a multiarrangement $(\CA, \mu)$, if there exists a hyperplane $H_0 \in \CA$ such that both
$(\CA', \mu')$ and $(\CA'', \mu^*)$ belong to $\CIFM$, and $\exp (\CA'', \mu^*) \subseteq \exp (\CA', \mu')$, 
then $(\CA, \mu)$ also belongs to $\CIFM$.
\end{itemize}
\end{defn}

\begin{remark}
	\label{rem:rank2indfree}
	As for simple arrangements, if $r(\CA) \le 2$,
	then $(\CA, \mu)$  is inductively free for any multiplicity $\mu$,  
	\cite[Cor.~7]{ziegler:multiarrangements}.
\end{remark}

The following is a useful tool for showing that a given 
multiarrangement is 
not inductively free by exhibiting a localization which fails to be  
inductively free.

\begin{theorem}
[{\cite[Thm.~1.3]{hogeroehrleschauenburg:free}}]
\label{thm:localmulti}
The class
$\CIFM$ is
closed under taking localizations.
\end{theorem}

We also require the fact that 
inductive freeness for 
multiarrangements behaves well with the 
product construction. 

\begin{theorem}
[{\cite[Thm.~1.4]{hogeroehrleschauenburg:free}}]
\label{thm:products}
A product of multiarrangements belongs to 
$\CIFM$ 
if and only if 
each factor belongs to $\CIFM$.
\end{theorem}

We recall a natural partial order on the 
set of multiplicities for a simple arrangement and the notion of an additively free multiarrangement from \cite[Def.~2.19]{hogeroehrle:ZieglerII}.

\begin{defn}
	\label{def:order}
Let $\CA$ be a fixed $\ell$-arrangement. 	
\begin{itemize}
	\item [(i)] 		For multiplicities $\mu_1$ and $\mu_2$ on $\CA$, define 
	$\mu_1 \le \mu_2$ provided $\mu_1(H) \le \mu_2(H)$ for every $H$ in $\CA$.
		\item [(ii)] 	
				The multiarrangement $(\CA, \mu)$ is said to be  
		\emph{additively free} if there is a \emph{free filtration} of multiplicities 
		$\mu_0 < \mu_1 < \cdots < \mu_n = \mu$,
		of $(\CA, \mu)$, i.e., where each $(\CA, \mu_i)$ is free with $|\mu_i| = i$. In particular, $(\CA, \mu_0) = \Phi_\ell$. 
		\end{itemize}
\end{defn}

Clearly, if $(\CA, \mu)$ is inductively free, it is additively free.

\section{Reflection arrangements and their restrictions}
\label{ssect:refl}

Let $W$ be a complex reflection group in $\GL(V)$. 
The \emph{reflection arrangement} $\CA = \CA(W)$ of $W$ in $V$ is 
the hyperplane arrangement 
consisting of the reflecting hyperplanes of the elements in $W$
acting as reflections in $V$.
If $\CA = \CA(W)$ is a reflection arrangement of 
the complex reflection group $W$, 
then $\CA$ is free, thanks to work of 
Terao \cite{terao:freereflections}.
Thus $(\CA'', \kappa)$ is also free, by 
Theorem \ref{thm:zieglermulti}.

The irreducible finite complex reflection groups were 
classified by Shephard and Todd, \cite{shephardtodd}.
Using the classification and nomenclature of 
\cite{shephardtodd},
we recall the classification of all inductively free reflection arrangements $\CA(W)$. 

\begin{theorem} [{\cite[Thm.~1.1, 1.2]{hogeroehrle:indfree}}]
	\label{thm:indfree1}
	For a finite complex reflection group  $W$,  
	let  $\CA = \CA(W)$ be its reflection arrangement. Then the following hold:
	\begin{itemize}
		\item[(i)]  $\CA$ is inductively free if and only if 
		$W$ does not admit an irreducible factor
		isomorphic to a monomial group 
		$G(r,r,\ell)$ for $r, \ell \ge 3$, 
		$G_{24}, G_{27}, G_{29}, G_{31}, G_{33}$, or $G_{34}$.
		\item[(ii)]  $\CA$ is inductively free if and only if $\CA$ is hereditarily inductively free.
	\end{itemize}
\end{theorem}

Thanks to work by 
Orlik--Terao \cite{orlikterao:free}, 
\cite[Prop.\ 6.73, Prop.\ 6.77, Cor.\ 6.86, App. D]{orlikterao:arrangements}
and Hoge--R\"ohrle \cite{hogeroehrle:free},
 for $W$ an irreducible complex reflection group, every restriction 
$\CA(W)^X$ of the reflection arrangement $\CA(W)$ 
is free.
So in particular, the  
Ziegler multiplicity of a restriction of the latter to 
a hyperplane is free, by Theorem~\ref{thm:zieglermulti}.

The following is immediate from Theorems \ref{thm:main} and \ref{thm:indfree1} and gives the reverse implication of part (i) of Theorem \ref{thm:kappa2}.

\begin{theorem} 
	\label{thm:indfreemain1}
	Let $W$ be a finite complex 
	reflection group with reflection arrangement 
	$\CA(W)$, let $X \in L(\CA(W))$ and  
	set $\CA = \CA(W)^X$.
	If $\CA(W)$ is inductively free, then so is $(\CA'', \kappa)$.
\end{theorem}

Orlik and Solomon defined intermediate 
arrangements $\CA^k_\ell(r)$ in 
\cite[\S 2]{orliksolomon:unitaryreflectiongroups}
(cf.\ \cite[\S 6.4]{orlikterao:arrangements}) which
interpolate between the
reflection arrangements of $G(r,r,\ell)$ and $G(r,1,\ell)$. 
These are the ones that  
show up as restrictions of the reflection arrangement
of $G(r,r,n)$, for some $n$,  
\cite[Prop.\ 2.14]{orliksolomon:unitaryreflectiongroups} 
(cf.~\cite[Prop.\ 6.84]{orlikterao:arrangements}),
see also \cite[Ex.~3.2]{amendhogeroehrle:indfree}.

For 
$r, \ell \geq 2$ and $0 \leq k \leq \ell$ the defining polynomial of
$\CA^k_\ell(r)$ is given by
$$Q(\CA^k_\ell(r)) = x_1 \cdots x_k\prod\limits_{\substack{1 \leq i < j \leq \ell\\ 0 \leq n < r}}(x_i - \zeta^nx_j),$$
where $\zeta$ is a primitive $r\th$ root of unity,
so that 
$\CA^\ell_\ell(r) = \CA(G(r,1,\ell))$ and 
$\CA^0_\ell(r) = \CA(G(r,r,\ell))$. We recall
\cite[Props.\ 2.11,  2.13]{orliksolomon:unitaryreflectiongroups}
(cf.~\cite[Props.~6.82,  6.85]{orlikterao:arrangements}):

\begin{proposition}
	\label{prop:intermediate}
	Let $\CA = \CA^k_\ell(r)$.
	\begin{enumerate}
		\item[(i)] $\CA$ is free with $\exp\CA = \{1, r + 1, \ldots, (\ell - 2)r + 1, (\ell - 1)r - \ell + k + 1\}$.
		\item[(ii)] Let $H \in \CA$. The type of $\CA^{H}$ is given in Table \ref{table2}.
	\end{enumerate}
\end{proposition}

\begin{table}[ht!b]
	\renewcommand{\arraystretch}{1.5}
	\begin{tabular}{llll}\hline
		$k$ & \multicolumn{2}{l}{$\alpha_{H}$} & Type of $\CA^{H}$\\ \hline
		$0$ & arbitrary & & $\CA^1_{\ell - 1}(r)$\\
		$1, \ldots, \ell - 1$ & $x_i - \zeta x_j$ & $1 \leq i < j \leq k < \ell$ & $\CA^{k - 1}_{\ell - 1}(r)$\\
		$1, \ldots, \ell - 1$ & $x_i - \zeta x_j$ & $1 \leq i \leq k < j \leq \ell$ & $\CA^k_{\ell - 1}(r)$\\
		$1, \ldots, \ell - 1$ & $x_i - \zeta x_j$ & $1 \leq k < i < j \leq \ell$ & $\CA^{k + 1}_{\ell - 1}(r)$\\
		$1, \ldots, \ell - 1$ & $x_i$ & $1 \leq i \leq \ell$ & $\CA^{\ell - 1}_{\ell - 1}(r)$\\
		$\ell$ & arbitrary & & $\CA^{\ell - 1}_{\ell - 1}(r)$\\ \hline
	\end{tabular}
	\bigskip
	\caption{Restriction types of $\CA^k_\ell(r)$}
	\label{table2}
\end{table}

Next we recall the classification
of all inductively free restrictions $\CA(W)^X$. 

\begin{theorem}
	[{\cite[Thm.~1.2, Thm.~1.3]{amendhogeroehrle:indfree}}]
	\label{thm:indfree2}
	Let $W$ be a finite, irreducible, complex 
	reflection group with reflection arrangement 
	$\CA = \CA(W)$ and let $X \in L(\CA)$.
	Then
	\begin{itemize}
		\item[(i)]   
		$\CA^X$ is inductively free 
		if and only if one of the following holds:
		\begin{itemize}
			\item[(a)] 
			$\CA$ is inductively free;
			\item[(b)] 
			$W = G(r,r,\ell)$ and 
			$\CA^X \cong \CA^k_p(r)$, where $p = \dim X$ and $p - 2 \leq k \leq p$; 
			\item[(c)] 
			$W$ is one of $G_{24}, G_{27}, G_{29}, G_{31}, G_{33}$, or $G_{34}$ and $X \in L(\CA) \setminus \{V\}$ with  $\dim X \leq 3$;
		\end{itemize}
		\item[(ii)]    $\CA^X$ is inductively free if and only if it is hereditarily inductively free.
	\end{itemize}
\end{theorem}

It is immediate from Theorem \ref{thm:main} that if 	$\CA = \CA(W)^X$ is inductively free, for $X \in L(\CA(W))$, then so is $(\CA'', \kappa)$.
Moreover, we have the following for low rank restrictions.

\begin{corollary}
	\label{cor:kappa1-restr2}
	Let $W$ be a finite, irreducible, complex 
	reflection group with reflection arrangement 
	$\CA(W)$ and let $V \ne X \in L(\CA(W))$ with $\dim X \le 3$.  
	Let $\CA := \CA(W)^X$. 
	Then $(\CA'', \kappa)$ is inductively free.
\end{corollary}

\begin{proof}
	Clearly, in each instance Remark~\ref{rem:rank2indfree} gives that 			
	$(\CA'', \kappa)$ is inductively free. 
	It follows from Theorems~\ref{thm:indfree1} and \ref{thm:indfree2}
	that the restricted arrangement $\CA$ is also inductively free. (In case 
	$W = G(r,r,\ell)$ with $r, \ell \ge 3$, we have $\CA(W)^X \cong \CA^k_3(r)$ for $1 \le k \le 2$.)  
\end{proof}

\section{Proof of Theorem \ref{thm:kappa2}}
\label{sec:proofs}

In order to classify all restrictions of reflection arrangements $\CA = \CA(W)^X$ so that $(\CA'', \kappa)$ is inductively free,
it follows from Theorems \ref{thm:main} and \ref{thm:indfreemain1}  
that 
we only need to investigate the instances when either 		$\CA(G(r,r,n))^X \cong \CA^k_\ell(r)$, where $\ell = \dim X$ and $1 \leq k \leq \ell-3$, or when  
$W$ is of type $G_{33}$, 
or $G_{34}$ and $X \in L(\CA(W)) \setminus \{V\}$ with  $\dim X \geq 4$.		
We consider these instances separately.

\subsection{The monomial groups $G(r,r,n)$}
\label{ssect:monomial2}
In this section we consider the restrictions of the reflection arrangements of the monomial groups $W = G(r,r,n)$ for $r, n \ge 3$.
Recall the intermediate arrangements $\CA_\ell^k(r)$ from above.

For $\CA = \CA_3^k(r)$ and $r \ge 3$, 
$\CA''$ is of rank $2$, so  $\CA''$ and $(\CA'', \kappa)$ are 
inductively free, by Remark \ref{rem:rank2indfree}.
So we may assume from now on that $\ell \ge 4$. 

It follows from Proposition \ref{prop:intermediate}(ii) that if we restrict $\CA^k_\ell(r)$ to a coordinate hyperplane, then we obtain 
$\CA^{\ell  - 1}_{\ell - 1}(r)  = \CA(G(r,1,\ell-1))$. 
We consider that case first, cf.~\cite[Ex.~1.14]{hogeroehrle:ZieglerII}.

\begin{corollary}
	\label{cor:akl1}
	Let $\CA = \CA_\ell^k(r)$ for $r \ge 3$, $\ell \ge 4$
	and $1 \le k \le \ell-1$. Fix a coordinate hyperplane $H_0$ in $\CA$.
	Then $(\CA'', \kappa)$ is inductively free. 
\end{corollary}

\begin{proof}
	It follows from Proposition \ref{prop:intermediate} that $\CA' \cong \CA^{k-1}_\ell(r)$ is free. Moreover, thanks to Theorem \ref{thm:indfreemain1}, $\left(\CA_\ell^k(r)\right)^{H_0} \cong \CA^{\ell  - 1}_{\ell - 1}(r)  = \CA(G(r,1,\ell-1))$ is inductively free. So the assertion follows from Theorem \ref{thm:main2}. 
\end{proof}

Next we consider the case when $H_0$ is not a coordinate hyperplane in $\CA$.
Here we shall see that $(\CA'', \kappa)$ is not always inductively free, in contrast to Corollary \ref{cor:akl1}.  
We begin with the instance when 
$\CA = \CA_\ell^k(r)$ itself is still inductively free. In these instances 
$(\CA'', \kappa)$ is inductively free which is immediate by Theorem \ref{thm:main}.

\begin{lemma}
	\label{lem:akl3}
	Let $\CA = \CA_\ell^k(r)$ for $r \ge 3$, $\ell \ge 4$
	and $k \in \{\ell-2, \ell-1\}$. Let $H_0 = H_{i,j}(\zeta)$.
	Then 	$(\CA'', \kappa)$ is inductively free.
\end{lemma}

Next we consider the instances when $\CA = \CA_\ell^k(r)$ is not inductively free, i.e., when $1 \le k \le \ell-3$, cf.~Theorem \ref{thm:indfree2}(i)(b) (cf.~\cite[Ex.~1.14]{hogeroehrle:ZieglerII}).

	\begin{lemma}
		\label{lem:akl2}
		Let $\CA = \CA_\ell^k(r)$ for $r \ge 3$, $\ell \ge 5$. Fix $H_0 = H_{s,t}(1) = \ker(x_s-x_t)$.
		In the following instances	$(\CA'', \kappa)$ is not inductively free:
		\begin{itemize}
			\item [(i)] $1 \le k \le \ell-3$  and $1 \le s < t \le \ell-3$;
			\item [(ii)] $1 \le k \le \ell-4$  and $1 \le s \le k < t$;
			\item [(iii)] $1 \le k \le \ell-5$  and $1 \le k < s < t$.
		\end{itemize}
	\end{lemma}
	
	\begin{proof}	
		Set $Y_{i,j}(\zeta) := H_0 \cap H_{i,j}(\zeta)$ and 
		$Y_i  := H_0 \cap H_i$ in $\CA''$.
		Then one readily checks that 
		\begin{equation}
		\label{eq:akl}
		\kappa(Y) = 
		\begin{cases}
		1   & \text{ for } Y = Y_{i,j}(\zeta) \text{ and }  \{s,t\} \cap \{i,j\} = \varnothing,\\
		1   & \text{ for } Y = Y_i \text{ and }  i \notin \{s,t\}.\\
		\end{cases}
		\end{equation}	
		First suppose we are in case (i). Define 
		\[
		Z := \bigcap_{\ell-2 \le i < j \le \ell}{H_{i,j}(\zeta)}
		%= \bigcap_{\ell-2 \le i \le \ell} H_i,
		\]
		which is of rank $3$ in $L(\CA)$ and let 
		\[
		X := H_0 \cap Z = H_0 \cap \left(\bigcap_{\ell-2 \le i < j \le \ell}{H_{i,j}(\zeta)}
		\right),
		\]
		which is of rank $3$ in $L(\CA'')$.
		According to \eqref{eq:akl}, the multiplicity 
		$\kappa_X$ of the localization  
		$\left( (\CA'')_X, \kappa_X\right)$ satisfies $\kappa_X \equiv \one$.
		Thus, it follows from the construction 
		and the hypotheses in (i) 
		that the localization 
		$((\CA'')_X, \kappa_X)$ is isomorphic to the simple 
		reflection arrangement $\CA(G(r,r,3))$.
		Owing to Theorem \ref{thm:indfree1}(i), 
		the latter is not inductively free.
		Therefore, $(\CA'', \kappa)$
		is not inductively free, by
		Theorem \ref{thm:localmulti}.
		
		Now suppose we are in case (ii) or (iii). Define 
		\[
		Z := \bigcap\limits_{\substack{1 \leq i < j \leq \ell\\ \{s,t\} \cap \{i,j\} = \varnothing}}{H_{i,j}(\zeta)},
		\]
		which is of rank $\ell-2$ in $L(\CA)$ and let 
		\[
		X := H_0 \cap Z, 
		\]
		which is of rank $\ell-2$ in $L(\CA'')$.
		According to \eqref{eq:akl}, the multiplicity 
		$\kappa_X$ of the localization  
		$\left( (\CA'')_X, \kappa_X\right)$ satisfies $\kappa_X \equiv \one$.
		Thus, it follows from the construction 
		and the hypotheses in (ii), respectively (iii), that the localization 
		$((\CA'')_X, \kappa_X)$ is isomorphic to the simple 
		reflection arrangement 
		\begin{equation}
		\label{eq:akl5}
		(\CA'')_X = 
		\begin{cases}
		\CA^{k-1}_{\ell-2}   & \text{ for } 1 \le k \le \ell-4 \text{ and }1 \le s \le k < t;\\
		\CA^{k}_{\ell-2}    & \text{ for } 	1 \le k \le \ell-5 \text{ and } 1 \le k < s < t.\\
		\end{cases}
		\end{equation}	
		It follows from \eqref{eq:akl5} and Theorem \ref{thm:indfree2}(i)(b) that $(\CA'')_X$ is not inductively free.
		Therefore, $(\CA'', \kappa)$
		is not inductively free either, by
		Theorem \ref{thm:localmulti}.	
	\end{proof}

So we are left to consider one instance when $\ell = 4$.

\begin{lemma}
	\label{lem:akl4}
	Let $\CA = \CA_4^1(r)$ for $r \ge 3$. Fix $H_0 = H_{i,j}(\zeta)$.
	Then  	$(\CA'', \kappa)$ is inductively free.
\end{lemma}

\begin{proof}
	We utilize Lemma \ref{lemma:free_sequence_to_ziegler_multiplicity}. For that purpose set $\CA_0 := \CA_4^0(r)$ and $\CA_1 := \CA_0 \cup \{H_1 = \ker x_1\} = \CA$.
	It follows from Theorems \ref{thm:kappa1}(ii), \ref{thm:indfree1}(i) and \ref{thm:main} that 
	$(\CA_0^{H_0}, \kappa_0) = (\CA_4^0(r)'', \kappa) = (\CA(G(r,r,4))'', \kappa)$
	is inductively free with exponents $\{r+1, 2r+1, 3r-3\}$ (cf.~Proposition \ref{prop:intermediate}(i)). By
	Lemma \ref{lemma:free_sequence_to_ziegler_multiplicity} and Proposition \ref{prop:intermediate}(ii), we observe that
	$\left((\CA_1^{H_0})^{H_0 \cap H_1}, \kappa_1^*\right) = \left((\CA_1^{H_1})^{H_0 \cap H_1}, \kappa\right) = \left(\CA(G(r,1,3))'', \kappa\right)$ which is inductively free, as it is of rank $2$, with exponents $\{r+1, 2r+1\}$. It follows from Definition \ref{def:indfree} that also $(\CA_1^{H_0}, \kappa_1) = (\CA_4^1(r)'', \kappa)$
	is inductively free.
\end{proof}

In view of Lemmas \ref{lem:akl3} and \ref{lem:akl2}, we still have to consider the following instances.

\begin{lemma}
	\label{lem:akl5}
		Let $\CA = \CA_\ell^k(r)$ for $r \ge 3$, $\ell \ge 5$. Fix $H_0 = H_{s,t}(1) = \ker(x_s-x_t)$. Let 
\begin{itemize}
	\item [(i)] $k \in \{\ell-4, \ell-3\}$  and $k < s < t$; or 
	\item [(ii)] $k = \ell-3$  and $s \le \ell-3 < t$.
\end{itemize}
Then  	$(\CA'', \kappa)$ is inductively free.
\end{lemma}

\begin{proof}
	It follows from Proposition \ref{prop:intermediate}(ii) that $\CA'' = \CA_{\ell-1}^{k+1}(r)$ in (i), and $\CA'' = \CA_{\ell-1}^{\ell-3}(r)$ in (ii).
	In any event, $\CA''$ is inductively free, by 
	Theorem \ref{thm:indfree2}(i)(b). In each case, we construct an induction table from the inductively free arrangement $(\CA'', \one)$ to $(\CA'', \kappa)$.
	
	Let $\CA_0=\CA^{\ell-4}_{\ell}(r),\CA_1=\CA^{\ell-4}_{\ell}(r)\cup\{H_1=\ker(x_{\ell-3})\}=\CA^{\ell-3}_{\ell}(r)$, then $\CA_0$ and $\CA_1$ are free, thanks to Proposition \ref{prop:intermediate}(i). Now use Lemma \ref{lemma:free_sequence_to_ziegler_multiplicity}(2) to see that for the triple $(\CA^{H_0}_0,\kappa_0),(\CA^{H_0}_1,\kappa_1), \left((\CA^{H_0}_1)^{H_0\cap H_1},\kappa_1^*\right)$ we get 
	\[
	\left((\CA^{H_0}_1)^{H_0\cap H_1},\kappa_1^*\right)=\left((\CA^{H_1}_1)^{H_0\cap H_1},\kappa\right)=(\CA^{\ell-1}_{\ell-1}(r)'',\kappa).
	\] The inclusion $\exp(\CA^{\ell-1}_{\ell-1}(r)'',\kappa)\subset\exp(\CA^{H_0}_0,\kappa_0)$ follows from Proposition \ref{prop:intermediate}(i) and Theorem \ref{thm:zieglermulti}. Now by Theorems \ref{thm:main} and \ref{thm:indfree1}(i) we see that $(\CA^{\ell-1}_{\ell-1}(r)'',\kappa)$ is inductively free. So it is sufficient to show the inductive freeness of $(\CA^{H_0}_0,\kappa_0)=(\CA^{\ell-3}_{\ell-1}(r),\kappa_0)$ to deduce part (i).
	
	Using Proposition \ref{prop:intermediate}(ii), we get $(\CA^{\ell-4}_{\ell}(r)^{H_0},\kappa)=(\CA^{\ell-3}_{\ell-1}(r),\kappa)$ and the Ziegler multiplicity is given by (with the notation as in Lemma \ref{lem:akl2}):\\
	\begin{equation*}
	\label{eq:akl47}
	\kappa(Y) =
	\begin{cases}
	r-1   & \text{ for } Y = Y_{s},\\
	2   & \text{ for } Y = Y_{i,t}(\zeta), i\neq s,\\
	2   & \text{ for } Y = Y_{s,j}(\zeta), j\neq t,\\
	1   & \text{ for } Y = Y_{i,j}(\zeta) \text{ and } \{s,t\} \cap \{i,j\} = \varnothing,\\
	1   & \text{ for } Y = Y_i, i\neq s.
	\end{cases}
	\end{equation*}
	
	Thanks to Theorem \ref{thm:indfree2}(i)(b), $(\CA^{\ell-3}_{\ell-1}(r),\mathbbm{1})$ is inductively free. So starting at the simple multiplicity $\mathbbm{1}$ we have to add $r-2$ hyperplanes of type $Y_{s}$ and $(\ell-2)r$ hyperplanes which are of type $Y_{i,t}(\zeta), (i\neq s)$ or $Y_{s,j}(\zeta),(j\neq t)$. We define multiplicities $\mu_i$ on $\CA^{\ell-3}_{\ell-1}(r)$ so that  $\mu_0=\mathbbm{1},\vert \mu_{i+1}\vert-\vert \mu_{i}\vert=1$ and $\mathbbm{1}=\mu_0\leq\dots\leq\mu_{(\ell-1)r-2}=\kappa$ as follows: Start by adding all the hyperplanes of type $Y_{1,t}(\zeta)$, then all of type $Y_{2,t}(\zeta)$ up to type $Y_{t-1,t}(\zeta)$. Continue by adding all of type $Y_{s,t+1}(\zeta)$ up to type $Y_{s,\ell}(\zeta)$. At the end add all the hyperplanes of type $Y_{s}$. 
	We present the induction in Table \ref{tableAl-4l} below.
	
	We now show that the $(\CA^{\ell-3}_{\ell-1}(r),\mu_i)$ provide a chain of inductively free arrangements from $(\CA^{\ell-3}_{\ell-1}(r),\mathbbm{1})$ to $(\CA^{\ell-3}_{\ell-1}(r),\kappa)$, proving part (i) of the lemma.
	
	While adding the hyperplanes of type $Y_{i,t}(\zeta),(i\neq s)$ and $Y_{s,j}(\zeta),(j\neq t)$ the Euler multiplicity coincides with the Ziegler multiplicity, giving $(\CA^{\ell-3}_{\ell-1}(r)'',\mu_i^*)=(\CA^{\ell-3}_{\ell-1}(r)'',\kappa)$ in every step for $1\leq i\leq (\ell-2)r$ which is inductively free, because of Theorems \ref{thm:main} and \ref{thm:indfree2}(i)(b). This is not affected by the different types of occurring restrictions (namely $\CA^{\ell-4}_{\ell-2}(r)$ and $\CA^{\ell-3}_{\ell-2}(r)$). 
	
	To show the equality between the multiplicities, let $H=Y_{i,t}(\zeta)$ be arbitrary and let $\mu_p$ be the unique multiplicity so that $\mu_p(H)=2$ and $\mu_{p-1}(H)=1$. There are three different types of localizations for $X\in\CA^{\ell-3}_{\ell-1}(r)^H$ in $\CA^{\ell-3}_{\ell-1}(r)$. The first kind has got $2$, the second kind $3$, and the third kind has got $r+2$ or $r+3$ elements, depending on the indices of the hyperplanes. The three types are as follows:
	\begin{itemize}
		\item[(a)] $\CA^{\ell-3}_{\ell-1}(r)_{X}=\{Y_{i,t}(\zeta),Y_{a,b}(\zeta')\}$ or $\{Y_{i,t}(\zeta),Y_{a}\}$ with $a,b\not\in\{i,t\}$, 
		 \item[(b)] $\CA^{\ell-3}_{\ell-1}(r)_{X}=\{Y_{i,t}(\zeta),Y_{i,a}(\zeta'),Y_{a,t}(\zeta'')\}$ or $\{Y_{i,t}(\zeta),Y_{i,j}(\zeta'),Y_{s,j}(\zeta'')\}$,
		 \item[(c)]  $\CA^{\ell-3}_{\ell-1}(r)_{X}=$ $\{Y_i,Y_{s},Y_{i,t}(1),Y_{i,t}(2),\dots, Y_{i,t}(r)\}$, respectively\\ $\CA^{\ell-3}_{\ell-1}(r)_{X}=\{Y_{s},Y_{s,j}(1),Y_{s,j}(2),\dots, Y_{s,j}(r)\}$.
	\end{itemize}
	In case of (a), using Proposition \ref{ATWEulerProp}(1) gives $\mu_p^*(X)=1=\kappa(X)$, since $(\mu_p)_{X}(Y_{a,b}(\zeta'))=(\mu_p)_{X}(Y_a)=1$. The cases (b) and (c) are handled similarly, using Proposition \ref{ATWEulerProp}(2). This gives the desired equality $(\CA^{\ell-3}_{\ell-1}(r)'',\mu_i^*)=(\CA^{\ell-3}_{\ell-1}(r)'',\kappa)$ at each step for $1\leq i\leq (\ell-2)r$. Note that because of Theorems \ref{thm:main} and \ref{thm:indfree2}(i)(b) every restriction $(\CA^{\ell-3}_{\ell-1}(r)'',\kappa)$ is inductively free. Now use Theorem \ref{thm:zieglermulti} to see 
	\[
	\exp(\CA^{\ell-3}_{\ell-1}(r)'',\mu_i^*)=\exp(\CA^{\ell-3}_{\ell-1}(r)'',\kappa)=\{r+1,\dots,(\ell-3)r+1,(\ell-2)r-1\},
	\]
	 which gives $\exp(\CA^{\ell-3}_{\ell-1}(r),\mu_{(\ell-2)r})=\{r+1,2r+1,\dots,$ $(\ell-2)r-1,(\ell-2)r+1\}$, by using Theorem \ref{thm:add-del} and this shows that $(\CA^{\ell-3}_{\ell-1}(r),\mu_{(\ell-2)r})$ is inductively free.
	
	Now $r-2$ hyperplanes of type $Y_{s}$ are left to be added. Owing to Theorem \ref{thm:kappa1},  for an arbitrary coordinate hyperplane $H_j\in\CA^{\ell-1}_{\ell-1}(r)$, the restriction $\left(\CA^{\ell-1}_{\ell-1}(r)^{H_j},\kappa\right)$ is inductively free with exponents $\exp\left(\CA^{\ell-1}_{\ell-1}(r)^{H_j},\kappa\right)=\{r+1,2r+1,\dots,$ $(\ell-2)r+1\}$ and $(\CA^{\ell-3}_{\ell-1}(r),\kappa)$ is free with exponents $\exp(\CA^{\ell-3}_{\ell-1}(r),\kappa)=\{r+1,2r+1,\dots,(\ell-2)r+1,(\ell-1)r-3\}$. So it is sufficient to show that $\left(\CA^{\ell-3}_{\ell-1}(r)^{Y_{s}},\mu_i^*\right)=\left(\CA^{\ell-1}_{\ell-1}(r)^{H_j},\kappa\right)$ for every $i$ with $(\ell-2)r+1\leq i\leq (\ell-1)r-2$. \\
	Showing that $\left(\CA^{\ell-3}_{\ell-1}(r)^{Y_{s}},\mu_{(\ell-2)r+1}^*\right)=\left(\CA^{\ell-1}_{\ell-1}(r)^{H_j},\kappa\right)$ is once again straightforward, relying on Proposition \ref{ATWEulerProp}. Thanks to Definition \ref{def:Euler}
	and Remark \ref{rem:euler}, it is sufficient to show that all localizations $\left(\CA^{\ell-3}_{\ell-1}(r)_X,(\mu_i)_X\right)$ for any $X\in\CA^{\ell-3}_{\ell-1}(r)^{Y_{s}}$ and $(\ell-2)r+1\leq i\leq (\ell-1)r-2$ share a common exponent, which implies in turn that the Euler multiplicity does not change. This is obvious for the localizations with two elements, since Proposition \ref{ATWEulerProp}(1) does not rely on the multiplicity of $Y_{s}$. The remaining localizations are $X_1=\{Y_{s},Y_{a},Y_{a,t}(1),\dots,Y_{a,t}(r)\}$ and $X_2=\{Y_{s},Y_{s,j}(1),\dots,Y_{s,j}(r)\},(j\neq t)$. We have shown above that $\mu_{(\ell-2)r+1}^*(X_1)=r+1$ which implies that $\exp\left(X_1,\mu_{(\ell-2)r+1}\vert_{X_1}\right)=\{r+1,r+2\}$ since $\vert \mu_{(\ell-2)r+1}\vert_{X_1}\vert=2r+3$. We also have $\left(\CA_{X_1},\mu_{(\ell-1)r-2}\vert_{X_1}\right)=\left(\CA_{X_1},\kappa\vert_{X_1}\right)\simeq \left(\CA^1_3(r)^{H_{2,3}(\zeta)},\kappa\right)$, as $2$-arrangements which shows that $\exp(\CA_{X_1},\mu_{(\ell-1)r-2}\vert_{X_1})=\{r+1,2r-1\}$ and we see that $r+1$ is the common exponent. In particular, we get $\mu_i^*(X_1)=r+1$ for an arbitrary $(\ell-2)r+1\leq i\leq (\ell-1)r-2$. The argument for $X_2$ is identical with the only difference that $\left(\CA_{X_2},\mu_{(\ell-1)r-2}\vert_{X_2}\right)\simeq \left(\CA^0_3(r)^{H_{1,2}(\zeta)},\kappa\right)$
	 and the exponents change from $\{r+1,r+1\}$ to $\{r+1,2r-2\}$. This completes the proof. 

\begin{table}[ht!b] \small
	\renewcommand{\arraystretch}{1.3}
	\begin{tabular}{lll}
		\hline
		$\exp(\CA',\mu')$ & $\alpha_H$ & $\exp(\CA'',\mu^*)$ \\
		\hline
		\hline
		$\exp(\CA^{\ell-3}_{\ell-1}(r),\mathbbm{1})=$ & $Y_{1,t}(1)$ & $\exp(\CA^{\ell-3}_{\ell-1}(r)'',\kappa)$ \\
		$\{1,r+1,2r+1,\dots,(\ell-3)r+1,(\ell-2)r-1\}$\\
		$\{2,r+1,2r+1,\dots,(\ell-3)r+1,(\ell-2)r-1\}$ & $Y_{1,t}(2)$ & $\exp(\CA^{\ell-3}_{\ell-1}(r)'',\kappa)$ \\
		\vdots & \vdots &  \vdots \\
		$\{r,r+1,2r+1,\dots,(\ell-3)r+1,(\ell-2)r-1\}$ & $Y_{1,t}(r)$ & $\exp(\CA^{\ell-3}_{\ell-1}(r)'',\kappa)$ \\
		$\{r+1,r+1,2r+1,\dots,(\ell-3)r+1,(\ell-2)r-1\}$ & $Y_{2,t}(1)$ & $\exp(\CA^{\ell-3}_{\ell-1}(r)'',\kappa)$ \\
		\vdots & \vdots & \vdots \\
		$\{(t-1)r-1,r+1,2r+1,\dots,(\ell-3)r+1,(\ell-2)r-1\}$ & $Y_{t-1,t}(r)$ & $\exp(\CA^{\ell-3}_{\ell-1}(r)'',\kappa)$ \\
		$\{(t-1)r,r+1,2r+1,\dots,(\ell-3)r+1,(\ell-2)r-1\}$ & $Y_{s,t+1}(1)$ & $\exp(\CA^{\ell-3}_{\ell-1}(r)'',\kappa)$ \\
		\vdots & \vdots & \vdots \\
		$\{(\ell-2)r,r+1,2r+1,\dots,(\ell-3)r+1,(\ell-2)r-1\}$ & $Y_{s,\ell}(r)$ & $\exp(\CA^{\ell-3}_{\ell-1}(r)'',\kappa)$ \\
		$\{r+1,2r+1,\dots,(\ell-3)r+1,(\ell-2)r-1,(\ell-2)r+1\}$ & $Y_{s}$ & $\exp(\CA^{\ell-1}_{\ell-1}(r)'',\kappa)$ \\
		\vdots & \vdots & \vdots \\
		$\{r+1,2r+1,\dots,(\ell-3)r+1,(\ell-2)r+1,(\ell-1)r-4\}$ & $Y_{s}$ & $\exp(\CA^{\ell-1}_{\ell-1}(r)'',\kappa)$ \\
		$\exp(\CA^{\ell-3}_{\ell-1}(r),\kappa)=$\\ $\{r+1,2r+1,\dots,(\ell-3)r+1,(\ell-2)r+1,(\ell-1)r-3\}$\\
		\hline		
	\end{tabular}
	\medskip	
	\caption{Induction table for $\CA=\CA^{\ell-4}_{\ell}(r),(\CA^H,\kappa)=(\CA^{\ell-3}_{\ell-1}(r),\kappa)$}
	\label{tableAl-4l}
\end{table}

To show (ii) let $\CA=\CA^{\ell-3}_\ell(r)$ and $H_0=H_{\ell-3,t}(1)$ with $\ell-3<t\leq \ell$ (we get $s=\ell-3$ after a permutation of basis vectors). Let $\CA_0=\CA\backslash\{H_1=\ker(x_{\ell-3})\}$ and $\CA_1=\CA$, then $\CA_0\simeq \CA^{\ell-4}_\ell(r)$ and we are in the setting of (i) for $k=\ell-4$. Now use Lemma \ref{lemma:free_sequence_to_ziegler_multiplicity} as in the proof of (i) to see that $(\CA'',\kappa)=(\CA_1^{H_0},\kappa_1)$ is inductively free because $(\CA_0^{H_0},\kappa_0)$ and $\left((\CA^{H_0}_1)^{H_0\cap H_1},\kappa_1^*\right)=\left((\CA^{H_1}_1)^{H_0\cap H_1},\kappa\right)=(\CA^{\ell-1}_{\ell-1}(r)'',\kappa)$ are.
	\end{proof}
		
We conclude this subsection with the following observation.

\begin{remark}
	\label{rem:akl}
	Owing to the results above, for $\CA = \CA_\ell^k(r)$ with $r \ge 3$, $\ell \ge 5$
	and $1 \le k \le \ell-3$, it depends on the choice of 
	$H_0$ whether 
	$(\CA'', \kappa)$ 
	is inductively free or not (e.g., see Corollary \ref{cor:akl1} 
	and Lemma	\ref{lem:akl2}).
\end{remark}

\subsection{Groups of type $G_{33}$ and $G_{34}$}
\label{ssect:exceptional}
Here we consider the outstanding restriction of the reflection arrangements of $G_{33}$ and $G_{34}$  of dimension at least $4$.

Since the pointwise stabilizer $W_X$ of $X$ in $L(\CA(W))$ is itself a complex reflection group, by \cite[Thm.~1.5]{steinberg:differential}, 
following 
\cite[\S 3, App.]{orliksolomon:unitaryreflectiongroups} (cf.~\cite[\S 6.4, App.\ C]{orlikterao:arrangements}),
we may  label the $W$-orbit 
of $X \in L(\CA(W))$ by the type $T$ say,
of $W_X$; so we usually denote such a restriction $\CA(W)^X$ simply by the pair $(W,T)$ in Tables \ref{tableAA} and \ref{tableB}.

It follows from the deletion part of Theorem \ref{thm:add-del-simple} that if 
both $\CA$ and $\CA''$ are free and $\exp \CA'' \subset \exp \CA$, then also $\CA'$ is free.
If in this case in addition $\CA''$ is inductively free, then so is $(\CA'', \kappa)$, thanks to Theorem \ref{thm:main2}.
Thus the following is a consequence of Theorems \ref{thm:main2},   \ref{thm:add-del-simple},  and \ref{thm:indfree2}(i)(c),   
the fact that all restrictions of reflection arrangements are free,
and the data in \cite[Tables 10, 11]{orliksolomon:unitaryreflectiongroups}
(cf.~\cite[C.14 - C.17]{orlikterao:arrangements}).

\begin{lemma}
	\label{lem:g33g34}
	Let $W$ be  of type 
	$G_{33}$ or  $G_{34}$ 
	with reflection arrangement $\CA(W)$.
	Let $\CA = \CA(W)^X$ for $X \in L(\CA(W))$ with restriction $\CA^H$ as in Table \ref{tableAA}.
	Then $(\CA^H,\kappa)$ is inductively free for any $H \in \CA$.	
	\begin{table}[ht!b] 
		\renewcommand{\arraystretch}{1.5}
		\begin{tabular}{lllc}
			\hline
			$W$ & $\CA = \CA(W)^X$ & $\CA^H$ & $\rank \CA^H$  \\
			\hline
			\hline
			$G_{33}$ & $(G_{33},A_1)$ & $(G_{33},A_1^2)$ & $3$  \\
			\hline
			$G_{34}$ & $(G_{34},A_1^2)$ & $(G_{34},A_1^3)$  & $3$  \\
			
			& $(G_{34},A_2)$ & $(G_{34},A_1A_2)$  & $3$  \\
			\hline
		\end{tabular}
		\medskip
		\caption{Restrictions of $G_{33}$ and $G_{34}$, I} 
		\label{tableAA} 
	\end{table}
\end{lemma}

The remaining instances of restrictions of reflection arrangements of type $G_{33}$ or $G_{34}$ to be considered are listed in Table \ref{tableB}.
Theorem \ref{thm:main2} does not apply here, as  $\CA'$ is not free. 
However, note that in each  case when $\CA^H$ is of rank $3$,  
$\CA^H$ is inductively free, thanks to Theorem \ref{thm:indfree2}(i)(c).
In each of these instances we provide a suitable induction table from the simple arrangement $\CA^H$ to  the Ziegler multiplicity.

\begin{lemma}
	\label{lem:g33g34-2}
	Let $W$ be  of type 
	$G_{33}$ or  $G_{34}$ 
	with reflection arrangement $\CA(W)$.
	Let $\CA = \CA(W)^X$ for $X \in L(\CA(W))$ with restriction $\CA^H$ as in Table \ref{tableB}.
	Then  
	$(\CA^H,\kappa)$ is inductively free if and only if $\dim X = 4$. 
\begin{table}[ht!b] 
	\renewcommand{\arraystretch}{1.5}
	\begin{tabular}{lllc}
		\hline
		$W$ & $\CA = \CA(W)^X$ & $\CA^H$ & $\rank \CA^H$ \\
		\hline
		\hline
		$G_{33}$ & $(G_{33},A_1)$ & $(G_{33},A_2)$ & $3$  \\
		\hline
		$G_{34}$ & $(G_{34},A_1)$ & $(G_{34},A_1^2)$, $(G_{34},A_2)$  & $4$ \\
		& $(G_{34},A_1^2)$ & $(G_{34},A_1A_2)$, $(G_{34},A_3)$  & $3$ \\
		& $(G_{34},A_2)$ & $(G_{34},A_3)$, $(G_{34},G(3,3,3))$  & $3$ \\ 				  
		\hline
	\end{tabular}
	\medskip
	\caption{Restrictions of $G_{33}$ and $G_{34}$, II}
	\label{tableB} 
\end{table}
\end{lemma}

To not interrupt the present flow, we defer the proof of the lemma to the appendix.

\subsection{Proof for Theorem \ref{thm:kappa2}}
We finally present the proof of our main theorem. Thanks to Theorems \ref{thm:main} and \ref{thm:indfree2},  
for the reverse implication of Theorem \ref{thm:kappa2}
we only need to investigate the instances when either $\CA(G(r,r,n))^X \cong \CA^k_\ell(r)$, where $\ell = \dim X$ and $1 \leq k \leq \ell-3$, or when  
$W$ is of type $G_{33}$, 
or $G_{34}$ and $X \in L(\CA(W)) \setminus \{V\}$ with  $\dim X \geq 4$.

The reverse implication in part (i) of Theorem \ref{thm:kappa2} follows from Theorem \ref{thm:indfreemain1}.

The reverse implication in part (ii) of Theorem \ref{thm:kappa2} follows from
Corollary \ref{cor:akl1}, 
and Lemmas \ref{lem:akl3}, \ref{lem:akl4}, and \ref{lem:akl5} along with  
  Theorem \ref{thm:indfree2}(b) and Proposition \ref{prop:intermediate}(b). 
While the forward implication in part (ii) follows from Lemma \ref{lem:akl2} together with  Theorem \ref{thm:indfree2}(b) and Proposition \ref{prop:intermediate}(b).

For part (iii) of Theorem \ref{thm:kappa2}, the equivalence follows from 
Lemmas \ref{lem:g33g34} and \ref{lem:g33g34-2}.

For the forward implication of Theorem \ref{thm:kappa2}, assume that 
we are not in case (i). Then by Theorem \ref{thm:indfree2}, we are in case (ii) or case (iii). So we are done.

\section{Proof of Theorem \ref{thm:delta-indfree}}
\label{sec:proof}

We require the following basic observation.

\begin{lemma}
	\label{lem:delta-restriction}
	Let 	$H_0 \in \CA$, $m_0 \ge 1$, and 
	$\delta = \delta_{H_0,m_0}$. 
	%is the multiplicity concentrated at $H_0$.
	Fix $H \in \CA \setminus \{H_0\}$.
	Then $(\CA^H, \delta^*) = (\CA^H, \delta_{H_0\cap H,m_0})$, i.e.~$\delta^*$ is itself a concentrated multiplicity on $\CA^H$.
\end{lemma}

\begin{proof}
	Let $X \in \CA^H$. If $H_0 \notin \CA_X$, then $\delta_X \equiv \one$, so $\delta^*(X) = 1$. If $H_0 \in \CA_X$, then $(\CA_X, \delta_X) = (\CA_X, \delta_{H_0, m_0})$ with $\exp(\CA_X, \delta_X) = \{m_0, e\}$ and $\exp(\CA_X\setminus\{H\}, \delta_X') = \{m_0, e-1\}$. Thus, $\delta^*(X) = m_0$, by  Remark \ref{rem:euler}.	
\end{proof}

\begin{proof}[Proof of Theorem \ref{thm:delta-indfree}]
The forward implication is \cite[Cor.~5.6]{hogeroehrle:ZieglerII}.

For the reverse implication assume that $(\CA, \delta)$ is inductively free.
We argue by induction on $|\CA|$. If $|\CA| \le 3$, then $\CA$ itself is inductively free. So suppose that $|\CA| > 3$ and that the result holds for arrangements with concentrated multiplicities with fewer hyperplanes.

We argue further by induction on $m_0 \ge 1$. If $m_0 = 1$, there is nothing to show. So assume $m_0 > 1$ and that the result holds for a concentrated multiplicity on $\CA$ with a lower parameter.

Since $(\CA, \delta)$ is inductively free, there is a hyperplane $H \in \CA$ so that the triple  $(\CA, \delta), (\CA', \delta')$ and  $(\CA'', \delta^*)$ with respect to $H$ consists of inductively free members obeying the containment condition for the exponents from Theorem \ref{thm:add-del}, cf.~Definition \ref{def:indfree}. 

If $H = H_0$, then $(\CA', \delta') = (\CA, \delta_{H_0,m_0-1})$, and so $\CA$ is inductively free, by induction on $m_0$.

If $H \ne H_0$, then $|\CA'| <  |\CA|$ and $|\CA''| <  |\CA|$. By Lemma \ref{lem:delta-restriction}, $\delta^*$ is again a concentrated multiplicity on $\CA''$. So it follows from induction on $|\CA|$ that both $\CA'$ and $\CA''$ are inductively free and thanks to 
Theorems \ref{thm:delta-free}, \ref{thm:add-del-simple}, \ref{thm:add-del} and Definitions \ref{def:indfree-simple} and \ref{def:indfree}, $\CA$ is inductively free.
\end{proof}

\begin{remark}
	\label{rem:delta}
	In the reverse implication of Theorem \ref{thm:delta-indfree}, using Theorem \ref{thm:main} and Proposition \ref{prop:delta}, one can show that
	any inductive chain of  $(\CA,\delta)$ descends to  one for $\CA$.
\end{remark}

\appendix
\section*{Appendix: Proof of Lemma \ref{lem:g33g34-2}} 
\renewcommand{\thesection}{A}

(1). Firstly, in each instance in Table \ref{tableB}  when $\CA$ is of rank $4$, we present an induction table for $(\CA'', \kappa)$; see Tables \ref{tableG33A1},
\ref{tableG34A_1^2}, and \ref{tableG34A2} below.

Let $\zeta=e^{\frac{2 \pi i}{3}}$ be a primitive third root of unity. We fix the order of the linear forms and refer to the corresponding hyperplane by its position in the defining polynomial of $\CA^H$ in each case. 

The defining polynomial of $(G_{33},A_1)$ is given by
{\tiny\begin{align*}
	Q(G_{33},A_1)=\ &(x_1 - x_4)
	(x_3 - x_4)
	(x_3 + (-\zeta^2)x_4)
	(2x_1 - x_2 + x_3 + x_4)
	(x_1 - x_3)
	(x_1 -\zeta x_3)
	(x_3 -\zeta x_4)
	(x_1 -\zeta x_4)\\
	&(2x_1 - x_2 + \zeta x_3 + \zeta^2 x_4)
	(x_1 -\zeta^2x_4)
	(x_1 -\zeta^2x_3)
	(-\zeta x_1 - x_2 + \zeta x_3 + x_4)
	(2x_1 - x_2 + \zeta^2x_3 + \zeta x_4)\\
	&(-\zeta x_1 - x_2 + x_3 + \zeta x_4)
	(-\zeta^2 x_1 - x_2 + x_3 + \zeta^2 x_4)
	(-\zeta^2x_1 - x_2 + \zeta^2x_3 + x_4)
	((-\zeta - 2\zeta^2)x_1)\\
	&(-\zeta^2x_1  -\zeta^2x_2 + \zeta^2x_3 + \zeta^2x_4)
	(-\zeta x_1 - x_2 + \zeta^2x_3 + \zeta^2x_4)
	(-\zeta^2x_1 - x_2 + \zeta x_3 + \zeta x_4)\\
	&(-\zeta^2 x_1 -\zeta^2x_2 + x_3 + \zeta x_4)
	(2x_1 -\zeta x_2 + x_3 + \zeta x_4)
	(2x_1 -\zeta^2x_2 + x_3 + \zeta^2 x_4)
	(2x_1 -\zeta x_2 + \zeta^2 x_3 + \zeta^2x_4)\\
	&(-\zeta^2 x_1 -\zeta^2 x_2 + \zeta x_3 + x_4)
	(2x_1 -\zeta x_2 + \zeta x_3 + x_4)
	(2x_1 -\zeta^2x_2 + \zeta^2x_3 + x_4)
	(2x_1 -\zeta^2x_2 + \zeta x_3 + \zeta x_4).
	\end{align*}
}

The Ziegler restriction of $(G_{33},A_1)$ with respect to $H=\ker(x_1-x_2)$ gives  $((G_{33},A_2),\kappa)$ with defining polynomial
{\tiny\begin{align*}
	Q((G_{33},A_2),\kappa)=\ &(x_2 - x_3)^2
	(x_2 -\zeta^2x_3)^2
	(-x_1 + x_2 + 3x_3)
	(x_2  -\zeta x_3)^2
	((-2\zeta - \zeta^2)x_3)^3
	(-x_1 + \zeta x_2 + (-2\zeta - \zeta^2)x_3)^2\\
	&(-x_1 + \zeta^2 x_2 + (-\zeta - 2\zeta^2)x_3)^2
	(-x_1 + x_2)^3
	(-x_1 + \zeta^2x_2 + (-\zeta + \zeta^2)x_3)^2
	(-x_1 + \zeta x_2 + (\zeta - \zeta^2)x_3)^2\\
	&((-\zeta^2)x_1 + x_2 + (\zeta - \zeta^2)x_3)^2
	((-\zeta^2)x_1 + \zeta x_2 + (-\zeta - 2\zeta^2)x_3)^2
	((-\zeta)x_1 + \zeta x_2 + 3x_3)
	((-\zeta^2)x_1 + \zeta^2x_2 + 3x_3).
	\end{align*}}
\bigskip

\begin{table}[ht!b] 
	\renewcommand{\arraystretch}{1}
	\begin{tabular}{lll}
		\hline
		$\exp(\CA',\mu')$ & $\alpha_H$ & $\exp(\CA'',\mu^*)$ \\
		\hline
		\hline
		$\exp((G_{33},A_2),\mathbbm{1})=\{1,6,7\}$ & $\alpha_5$ & $\{6,7\}$ \\
		$\{2,6,7\}$ & $\alpha_5$ & $\{6,7\}$ \\
		$\{3,6,7\}$ & $\alpha_8$ & $\{6,7\}$ \\
		$\{4,6,7\}$ & $\alpha_8$ & $\{6,7\}$ \\
		$\{5,6,7\}$ & $\alpha_4$ & $\{6,7\}$ \\
		$\{6,6,7\}$ &  $\alpha_{11}$ & $\{6,7\}$ \\
		$\{6,7,7\}$ &  $\alpha_7$ & $\{6,7\}$ \\
		$\{6,7,8\}$ &  $\alpha_{12}$ & $\{7,8\}$ \\
		$\{7,7,8\}$ &  $\alpha_{10}$& $\{7,8\}$ \\
		$\{7,8,8\}$ &  $\alpha_9$ & $\{7,8\}$ \\
		$\{7,8,9\}$ &  $\alpha_6$ & $\{7,9\}$ \\
		$\{7,9,9\}$ &  $\alpha_2$ & $\{7,9\}$ \\
		$\{7,9,10\}$ & $\alpha_1$ & $\{7,9\}$ \\
		$\exp((G_{33},A_2),\kappa)=\{7,9,11\}$ \\
		\hline
	\end{tabular}
	\medskip
	
	\caption{Induction table for $\CA=(G_{33},A_1),(\CA^H,\kappa)=((G_{33},A_2),\kappa)$}
	\label{tableG33A1}
\end{table}

\medskip

The defining polynomial of $(G_{34},A_1^2)$ is given by
{\tiny\begin{align*}
	Q(G_{34},A_1^2)=\ &(x_1-x_4)
	((\zeta + 2\zeta^2)x_4)
	(2x_1+\zeta^2 x_2+\zeta x_3-\zeta x_4)
	(x_2-\zeta x_3)
	(x_1-\zeta^2 x_4)
	(x_1-\zeta x_4)
	(2x_1+x_2+x_3-\zeta x_4)\\
	&(-\zeta x_1+\zeta^2 x_2+\zeta x_3+ 2x_4)
	(2x_1+\zeta^2 x_2+\zeta x_3+2\zeta x_4)
	(-\zeta x_1+ \zeta^2 x_2+\zeta x_3-x_4)
	(-\zeta^2 x_1+\zeta^2 x_2+\zeta x_3+ 2\zeta^2 x_4)\\
	&(-\zeta^2 x_1+\zeta^2 x_2+\zeta x_3-\zeta^2 x_4)
	(-\zeta x_1+x_2+x_3+2 x_4)
	(2x_1+x_2+x_3+ 2\zeta x_4)
	(-\zeta x_1+x_2+x_3-x_4)\\
	&(-\zeta^2 x_1+x_2+x3+2\zeta^2 x_4)
	(-\zeta^2 x_1+x_2+x_3-\zeta^2 x_4)
	((-\zeta - 2\zeta^2) x_1)
	(-\zeta^2 x_1+\zeta x_2+x_3-x_4)\\
	&(-\zeta^2 x_1+\zeta^2 x_2+ \zeta^2 x_3-x_4)
	(x_1-\zeta x_2)
	(x_1-\zeta^2 x_3)
	(-\zeta x_1+\zeta^2 x_2+x_3+ 2\zeta^2 x_4)
	(2x_1+x_2+ \zeta^2 x_3-x_4)\\
	&(2x_1+\zeta x_2+x_3+ 2\zeta^2 x_4)
	(2x_1+\zeta x_2+\zeta x_3-x_4)
	(2 x_1+\zeta^2 x_2+\zeta^2 x_3+ 2\zeta^2 x_4)
	(x_2 -\zeta^2 x_4)
	(x_3 -\zeta x_4)\\
	&(2x_1+x_2+ \zeta x_3-\zeta^2 x_4)
	(-\zeta^2 x_1+ \zeta x_2+x_3+2x_4)
	(-\zeta^2 x_1+ \zeta^2 x_2+ \zeta^2 x_3+ 2x_4)
	(-\zeta x_1+ \zeta^2 x_2+x_3 -\zeta^2 x_4)\\
	&(2x_1+x_2+ \zeta^2 x_3+2x_4)
	(2x_1+\zeta x_2+x_3 -\zeta^2 x_4)
	(2x_1+\zeta x_2+\zeta x_3+2x_4)
	(2x_1+\zeta^2 x_2+\zeta^2 x_3 -\zeta^2 x_4)
	(x_2-\zeta x_4)\\
	&(x_3-x_4)
	(2x_1+\zeta^2 x_2+x_3-x_4)
	(-\zeta^2 x_1+ \zeta^2 x_2+x_3-\zeta x_4)
	(-\zeta x_2+x_2+\zeta x_3+ 2\zeta x_4)
	(x_1-\zeta^2 x_2)
	(x_1-x_3)\\
	&(x_2-x_4)
	(x_3-\zeta^2x_4)
	(2x_1+x_2+\zeta x_3+ 2\zeta^2 x_4)
	(2x_1+\zeta^2 x_2+x_3+2x_4)
	(-\zeta^2 x_1+\zeta^2x_2+x_3+ 2\zeta x_4)\\
	&(-\zeta x_1+x_2+\zeta x_3-\zeta x_4)
	(x_1-x_2)
	(x_1-\zeta x_3)
	(2x_1+\zeta x_2+\zeta^2 x_3 -\zeta x_4)
	(2x_1+ \zeta x_2+ \zeta^2 x_3+ 2\zeta x_4)
	(x_2-x_3)\\
	&(x_2 -\zeta^2 x_3).
	\end{align*}}

The Ziegler restriction of  $(G_{34},A_1^2)$ with respect to $H=\ker((\zeta + 2\zeta^2)x_4)$ gives  $((G_{34},A_1 A_2),\kappa)$ with defining polynomial
{\tiny\begin{align*}
	Q((G_{34},A_1A_2),\kappa)=\ & (x_1)^4
	(2x_1 + \zeta^2x_2 + \zeta x_3)^2
	(x_2 - \zeta x_3)
	(2x_1 + x_2 + x_3)^2
	(-\zeta x_1 + \zeta^2x_2 + \zeta x_3)^2
	(-\zeta^2x_1 + \zeta^2x_2 + \zeta x_3)^2\\
	&(-\zeta x_1 + x_2 + x_3)^2
	(-\zeta^2x_1 + x_2 + x_3)^2
	(-\zeta^2x_1 + \zeta x_2 + x_3)^2
	(-\zeta^2x_1 + \zeta^2x_2 + \zeta^2x_3)^2
	(x_1 -\zeta x_2)
	(x_1  -\zeta^2x_3)\\
	&(-\zeta x_1 + \zeta^2x_2 + x_3)^2
	(2x_1 + x_2 + \zeta^2x_3)^2
	(2x_1 + \zeta x_2 + x_3)^2
	(2x_1 + \zeta x_2 +  \zeta x_3)^2
	(2x_1 + \zeta^2x_2 + \zeta^2x_3)^2
	(x_2)^3\\
	&(x_3)^3
	(2x_1 + x_2 + \zeta x_3)^2
	(2x_1 + \zeta^2x_2 + x_3)^2
	(-\zeta^2x_1 + \zeta^2x_2 + x_3)^2
	(-\zeta x_1 + x_2 + \zeta x_3)^2
	(x_1 -\zeta^2 x_2)
	(x_1 - x_3)\\
	&(x_1 - x_2)
	(x_1 -\zeta x_3)
	(2x_1 + \zeta x_2 + \zeta^2x_3)^2
	(x_2 - x_3)
	(x_2 -\zeta^2x_3).
	\end{align*}}

The Ziegler restriction of  $(G_{34},A_1^2)$ with respect to $H=\ker(x_1-x_4)$ gives  $((G_{34},A_3),\kappa)$ with defining polynomial
{\tiny\begin{align*}
	Q((G_{34},A_3),\kappa)=\ &((\zeta + 2\zeta ^2) x_3)^4
	(\zeta ^2 x_1+\zeta  x_2 +(-3\zeta  - 2\zeta ^2)x_3)^2
	(x_1 -\zeta  x_2)
	(x_1+x_2+ (-3\zeta  - 2\zeta ^2)x_3)^2\\
	&(\zeta^2 x_1+\zeta x_2-2\zeta ^2 x_3)^3
	(\zeta ^2x_1+\zeta x_2+\zeta^2 x_3)^4
	(x_1+x_2-2\zeta^2 x_3)^3
	(x_1+x_2+ \zeta^2 x_3)^4
	(-\zeta x_1+x_3)^2
	(-\zeta^2 x_1+x_3)^2\\
	&(\zeta^2 x_1+x_2+ (-\zeta + 2\zeta^2)x_3)^2
	(\zeta x_1+x_2+ (-2\zeta - 3\zeta^2)x_3)^2
	(\zeta^2x_1+ \zeta^2x_2+ (-2\zeta - 3\zeta^2)x_3)^2
	(\zeta^2 x_1+x_2+x_3)^4\\
	&(x_1+ \zeta^2 x_2+ 4 x_3)
	(\zeta x_1+\zeta x_2+4 x_3)
	(x_1-\zeta x_3)^2
	(x_2-x_3)^2
	(x_1-x_3)^2
	(x_2- \zeta^2 x_3)^2
	(x_1+\zeta x_2-2\zeta x_3)^3\\
	&(\zeta^2 x_1+x_2+4x_3)
	(\zeta^2 x_1+x_2+ (2\zeta - \zeta^2)x_3)^2
	(x_1-x_2)
	(x_1-\zeta^2 x_2). 
	\end{align*}}

\begin{table}[ht!b] \tiny
	\renewcommand{\arraystretch}{1}
	\begin{tabular}[t]{lll}
		\hline
		$\exp(\CA',\mu')$ & $\alpha_H$ & $\exp(\CA'',\mu^*)$ \\
		\hline
		\hline
		$\exp((G_{34},A_1 A_2),\mathbbm{1})=\{1,13,16\}$ & $\alpha_{18}$ & $\{13,16\}$ \\
		$\{2,13,16\}$ & $\alpha_{18}$ & $\{13,16\}$ \\
		$\{3,13,16\}$ & $\alpha_{22}$ & $\{13,16\}$ \\
		$\{4,13,16\}$ & $\alpha_9$ & $\{13,16\}$ \\
		$\{5,13,16\}$ & $\alpha_8$ & $\{13,16\}$ \\
		$\{6,13,16\}$ &  $\alpha_7$ & $\{13,16\}$ \\
		$\{7,13,16\}$ & $\alpha_{13}$  & $\{13,16\}$ \\
		$\{8,13,16\}$ & $\alpha_6$  & $\{13,16\}$ \\
		$\{9,13,16\}$ & $\alpha_{10}$  & $\{13,16\}$ \\
		$\{10,13,16\}$ & $\alpha_{23}$  & $\{13,16\}$ \\
		$\{11,13,16\}$ &  $\alpha_5$ & $\{13,16\}$ \\
		$\{12,13,16\}$ & $\alpha_{19}$  & $\{13,16\}$ \\
		$\{13,13,16\}$ &  $\alpha_{19}$ & $\{13,16\}$ \\
		$\{13,14,16\}$ &  $\alpha_1$ & $\{13,16\}$ \\
		$\{13,15,16\}$ &  $\alpha_1$ & $\{13,16\}$ \\
		$\{13,16,16\}$ &  $\alpha_1$ & $\{13,16\}$ \\
		$\{13,16,17\}$ &   $\alpha_{28}$ & $\{13,17\}$ \\
		$\{13,17,17\}$ &   $\alpha_{21}$ & $\{13,17\}$ \\
		$\{13,17,18\}$ &   $\alpha_{20}$ & $\{13,17\}$ \\
		$\{13,17,19\}$ &   $\alpha_{16}$ & $\{13,19\}$ \\
		$\{13,18,19\}$ &   $\alpha_{14}$ & $\{13,19\}$ \\
		$\{13,19,19\}$ &   $\alpha_{17}$ & $\{13,19\}$ \\
		$\{13,19,20\}$ &   $\alpha_{15}$ & $\{13,19\}$ \\
		$\{13,19,21\}$ &   $\alpha_4$ & $\{13,19\}$ \\
		$\{13,19,22\}$ &   $\alpha_2$ & $\{13,19\}$ \\
		$\exp((G_{34},A_1 A_2),\kappa)=\{13,19,23\}$\\
		\hline
	\end{tabular}
	\medskip
	\medskip
	\begin{tabular}[t]{lll}
		\hline
		$\exp(\CA',\mu')$ & $\alpha_H$ & $\exp(\CA'',\mu^*)$ \\
		\hline
		\hline
		$\exp((G_{34},A_3),\mathbbm{1})=\{1,11,13\}$ & $\alpha_{23}$ & $\{11,13\}$ \\
		$\{2,11,13\}$ & $\alpha_{20}$ & $\{11,13\}$ \\
		$\{3,11,13\}$ & $\alpha_{19}$& $\{11,13\}$ \\
		$\{4,11,13\}$ & $\alpha_{18}$& $\{11,13\}$ \\
		$\{5,11,13\}$ &$\alpha_{17}$ & $\{11,13\}$ \\
		$\{6,11,13\}$ & $\alpha_{13}$ & $\{11,13\}$ \\
		$\{7,11,13\}$ & $\alpha_{12}$ & $\{11,13\}$ \\
		$\{8,11,13\}$ & $\alpha_{11}$ & $\{11,13\}$ \\
		$\{9,11,13\}$ & $\alpha_{10}$ & $\{11,13\}$ \\
		$\{10,11,13\}$ & $\alpha_9$ & $\{11,13\}$ \\
		$\{11,11,13\}$ & $\alpha_4$ & $\{11,13\}$ \\
		$\{11,12,13\}$ & $\alpha_2$ & $\{11,13\}$ \\
		$\{11,13,13\}$ & $\alpha_{14}$ & $\{13,13\}$ \\
		$\{12,13,13\}$ & $\alpha_8$ & $\{13,13\}$ \\
		$\{13,13,13\}$ & $\alpha_6$ & $\{13,13\}$ \\
		$\{13,13,14\}$ & $\alpha_1$ & $\{13,13\}$ \\
		$\{13,13,15\}$ & $\alpha_{14}$ & $\{13,13\}$ \\
		$\{13,13,16\}$ & $\alpha_{14}$ & $\{13,16\}$ \\
		$\{13,14,16\}$ & $\alpha_8$ & $\{13,14\}$ \\
		$\{13,14,17\}$ & $\alpha_8$ & $\{13,17\}$ \\
		$\{13,15,17\}$ & $\alpha_6$ & $\{13,15\}$ \\
		$\{13,15,18\}$ & $\alpha_6$ & $\{13,18\}$ \\
		$\{13,16,18\}$ & $\alpha_1$ & $\{13,16\}$ \\
		$\{13,16,19\}$ &$\alpha_1$ & $\{13,19\}$ \\
		$\{13,17,19\}$ & $\alpha_{21}$ & $\{13,19\}$ \\
		$\{13,18,19\}$ & $\alpha_{21}$ & $\{13,19\}$ \\
		$\{13,19,19\}$ & $\alpha_7$ & $\{13,19\}$ \\
		$\{13,19,20\}$ & $\alpha_5$ & $\{13,19\}$ \\
		$\{13,19,21\}$ & $\alpha_7$ & $\{13,19\}$ \\
		$\{13,19,22\}$ & $\alpha_5$ & $\{13,19\}$ \\
		$\exp((G_{34},A_3),\kappa)=\{13,19,23\}$\\
		\hline
	\end{tabular}
	
	\medskip
	
	\medskip
	
	\caption{Induction tables for $\CA=(G_{34},A_1^2),(\CA^H,\kappa)=((G_{34},A_1A_2),\kappa),((G_{34},A_3),\kappa)$}
	\label{tableG34A_1^2}
\end{table}

The defining polynomial of $(G_{34},A_2)$ is given by
{\tiny\begin{align*}
	Q(G_{34},A_2)=\ &(x_3 - x_4)
	(x_3 -\zeta^2x_4)
	(\zeta^2x_1 + \zeta x_2 + x_3 + 3x_4)
	(x_1 -\zeta x_2)
	(x_3 -\zeta x_4)
	((-2\zeta - \zeta^2)x_4)
	(\zeta^2x_1 + \zeta x_2 + \zeta x_3 +(-2\zeta - \zeta^2)x_4)\\
	&(x_1 + x_2 + x_3 + 3x_4)
	(\zeta^2x_1 + \zeta x_2 + \zeta^2 x_3 + (-\zeta - 2\zeta^2)x_4)
	(\zeta^2x_1 + \zeta x_2 + x_3)
	(x_1 + x_2 + \zeta x_3 + (-2\zeta - \zeta^2)x_4)\\
	&(x_1 + x_2+ \zeta^2x_3 + (-\zeta - 2\zeta^2)x_4)
	(x_1 + x_2 + x_3)
	(\zeta^2x_1 + \zeta x_2 + \zeta^2 x_3 + (-\zeta + \zeta^2)x_4)
	(\zeta^2x_1 + \zeta x_2 + \zeta x_3 + (\zeta - \zeta^2)x_4)\\
	&(x_1 + x_2+  \zeta^2x_3 + (-\zeta + \zeta^2)x_4)
	(x_1 + x_2 + \zeta x_3 + (\zeta - \zeta^2)x_4)
	(\zeta x_1 + x_2 + x_3 + (\zeta - \zeta^2)x_4)\\
	&(\zeta^2x_1 + \zeta^2x_2 + x_3 + (\zeta - \zeta^2)x_4)
	(-\zeta x_1 + x_4)
	(-\zeta^2 x_2 + x_4)
	(\zeta^2x_1 + x_2 + x_3 + (-\zeta + \zeta^2)x_4)\\
	&(\zeta x_1 + x_2 + \zeta x_3 + (-\zeta - 2\zeta^2)x_4)
	(\zeta^2x_1 + \zeta^2x_2 + \zeta x_3 +(-\zeta - 2\zeta^2)x_4)
	(\zeta^2x_1 + x_2 + \zeta^2x_3 + (-2\zeta - \zeta^2)x_4)\\
	&(x_1 + \zeta^2x_2 + \zeta x_3 + 3x_4)
	(\zeta x_1 + x_2 + \zeta^2x_3 + 3x_4)
	(\zeta x_1 + \zeta x_2 + \zeta x_3 + 3x_4)
	(\zeta^2x_1 + \zeta^2x_2 + \zeta^2x_3 + 3x_4)
	(x_1  -\zeta^2 x_3)\\
	&(x_2 -\zeta x_3)
	(x_1 - x_3)
	(x_2 -\zeta^2x_3)
	(x_1 + \zeta x_2 + \zeta^2 x_3 + 3x_4)
	(\zeta^2x_1 + x_2 + \zeta x_3)
	(x_1 -\zeta x_4)
	(x_2 - x_4)\\
	&(\zeta^2x_1 + x_2 + x_3 + (-\zeta 2\zeta^2)x_4)
	(x_1 - x_4)
	(x_2 + (-\zeta^2)x_4)
	(x_1 + \zeta x_2 + x_3 + (-2\zeta - \zeta^2)x_4)
	(x_1 -\zeta x_3)
	(x_2 - x_3)\\
	&(\zeta^2 x_1 + x_2 + \zeta x_3 + 3x_4)
	(\zeta^2x_1 + x_2 + \zeta^2x_3 + (\zeta - \zeta^2)x_4)
	(x_1 + \zeta x_2 + \zeta x_3 + (-\zeta + \zeta^2)x_4)
	(\zeta x_1 + \zeta^2x_2 + x_3 + 3x_4)\\
	&(x_1 - x_2)
	(x_1 -\zeta^2x_2).
	\end{align*}}

Now the Ziegler restriction of  $(G_{34},A_2)$ with respect to $H=\ker((-2\zeta - \zeta^2)x_4)$ gives  $((G_{34},G(3,3,3)),\kappa)$ with defining polynomial
{\tiny\begin{align*}
	Q((G_{34},G(3,3,3)),\kappa)=\ & (x_3)^3
	(\zeta^2 x_1 + \zeta x_2 + x_3)^4
	(x_1  -\zeta x_2)
	(\zeta^2 x_1 + \zeta x_2 + \zeta x_3)^3
	(x_1 + x_2 + x_3)^4
	(\zeta^2 x_1 + \zeta x_2 + \zeta^2 x_3)^3
	(x_1 + x_2 + \zeta x_3)^3\\
	&(x_1 + x_2 + \zeta^2 x_3)^3
	(-\zeta x_1)^3
	(-\zeta^2 x_2)^3
	(\zeta^2 x_1 + x_2 + x_3)^3
	(\zeta^2 x_1 + x_2 + \zeta^2 x_3)^3
	(x_1 -\zeta^2 x_3)
	(x_2 -\zeta x_3)
	(x_1 - x_3)\\
	&(x_2 -\zeta^2 x_3)
	(x_1 + \zeta x_2 + \zeta^2x_3)^4
	(x_1 -\zeta x_3)
	(x_2 - x_3)
	(x_1 - x_2)
	(x_1 -\zeta^2 x_2).\end{align*}}

The Ziegler restriction of  $(G_{34},A_2)$ with respect to $H=\ker(x_3 - x_4)$ gives  $((G_{34},A_3),\kappa)$ with defining polynomial
{\tiny\begin{align*}
	Q((G_{34},A_3),\kappa)=\ &((-\zeta - 2\zeta^2)x_3)^3
	(\zeta^2x_1 + \zeta x_2 + 4x_3)
	(x_1 -\zeta x_2)
	(\zeta^2x_1 + \zeta x_2 + x_3)^3
	(x_1 + x_2 + 4x_3)
	(x_1 + x_2 + x_3)^3\\
	&(\zeta^2x_1 + \zeta x_2 + (-\zeta + 2\zeta^2)x_3)^2
	(\zeta^2x_1 + \zeta x_2 + (2\zeta - \zeta^2)x_3)^2
	(x_1 + x_2 + (-\zeta + 2\zeta^2)x_3)^2
	(x_1 + x_2 + (2\zeta - \zeta^2)x_3)^2\\
	&(\zeta x_1 + x_2 + (-2\zeta^2)x_3)^2
	(\zeta^2x_1 + \zeta^2x_2 + (-2\zeta^2)x_3)^2
	(-\zeta x_1 + x_3)^2
	(-\zeta^2x_2 + x_3)^2
	(\zeta^2x_1 + x_2 -2\zeta x_3)^2
	(x_1 - x_3)^2\\
	&(x_2 -\zeta^2x_3)^2
	(x_1 + \zeta x_2 + (-3\zeta - 2\zeta^2)x_3)^2
	(\zeta^2x_1 + x_2 + \zeta x_3)^3
	(x_1 -\zeta x_3)^2
	(x_2 - x_3)^2\\
	&(\zeta^2x_1 + x_2 + (-2\zeta - 3\zeta^2)x_3)^2
	(\zeta x_1 + \zeta^2x_2 + 4x_3)
	(x_1 - x_2)
	(x_1 -\zeta^2x_2).
	 \end{align*}}
\medskip

\begin{table}[ht!b] \tiny
	\renewcommand{\arraystretch}{1}
	\begin{tabular}[t]{lll}
		\hline
		$\exp(\CA',\mu')$ & $\alpha_H$ & $\exp(\CA'',\mu^*)$ \\
		\hline
		\hline
		$\exp((G_{34},A_3),\mathbbm{1})=\{1,11,13\}$ & $\alpha_6$ & $\{11,13\}$ \\
		$\{2,11,13\}$ & $\alpha_6$ & $\{11,13\}$\\
		$\{3,11,13\}$ & $\alpha_{15}$ & $\{11,13\}$ \\
		$\{4,11,13\}$ & $\alpha_{11}$ & $\{11,13\}$ \\
		$\{5,11,13\}$ & $\alpha_{12}$ & $\{11,13\}$ \\
		$\{6,11,13\}$ & $\alpha_4$ & $\{11,13\}$ \\
		$\{7,11,13\}$ & $\alpha_4$ & $\{11,13\}$ \\
		$\{8,11,13\}$ & $\alpha_1$ & $\{11,13\}$ \\
		$\{9,11,13\}$ & $\alpha_1$ & $\{11,13\}$ \\
		$\{10,11,13\}$ & $\alpha_{19}$ & $\{11,13\}$ \\
		$\{11,11,13\}$ & $\alpha_{19}$ & $\{11,13\}$ \\
		$\{11,12,13\}$ &  $\alpha_{22}$& $\{12,13\}$ \\
		$\{12,12,13\}$ &  $\alpha_{21}$& $\{12,13\}$ \\
		$\{12,13,13\}$ &  $\alpha_{20}$& $\{12,13\}$ \\
		$\{12,13,14\}$ &  $\alpha_{18}$& $\{13,14\}$ \\
		$\{13,13,14\}$ &  $\alpha_{17}$& $\{13,14\}$ \\
		$\{13,14,14\}$ &  $\alpha_{16}$& $\{13,14\}$ \\
		$\{13,14,15\}$ &  $\alpha_{14}$& $\{13,15\}$ \\
		$\{13,15,15\}$ &  $\alpha_{13}$& $\{13,15\}$ \\
		$\{13,15,16\}$ &  $\alpha_{10}$& $\{13,16\}$ \\
		$\{13,16,16\}$ &  $\alpha_9$& $\{13,16\}$ \\
		$\{13,16,17\}$ &  $\alpha_8$& $\{13,16\}$ \\
		$\{13,16,18\}$ &  $\alpha_7$& $\{13,16\}$ \\
		$\exp((G_{34},A_3),\kappa)=\{13,16,19\}$\\
		\hline
	\end{tabular}
	\begin{tabular}[t]{ l l l }
		\hline
		$\exp(\CA',\mu')$ & $\alpha_H$ & $\exp(\CA'',\mu^*)$ \\
		\hline
		\hline		
		$\exp((G_{34},G(3,3,3)),\mathbbm{1})=\{1,7,13\}$ & $\alpha_{12}$ & $\{7,13\}$ \\
		$\{2,7,13\}$ & $\alpha_{11}$ & $\{7,13\}$ \\
		$\{3,7,13\}$ & $\alpha_{10}$ & $\{7,13\}$ \\
		$\{4,7,13\}$ & $\alpha_9$ & $\{7,13\}$ \\
		$\{5,7,13\}$ & $\alpha_8$ & $\{7,13\}$ \\
		$\{6,7,13\}$ & $\alpha_7$ & $\{7,13\}$ \\
		$\{7,7,13\}$ & $\alpha_6$ & $\{7,13\}$ \\
		$\{7,8,13\}$ & $\alpha_4$ & $\{7,13\}$ \\
		$\{7,9,13\}$ & $\alpha_1$ & $\{7,13\}$ \\
		$\{7,10,13\}$ & $\alpha_4$ & $\{10,13\}$ \\
		$\{8,10,13\}$ & $\alpha_7$ & $\{10,13\}$ \\
		$\{9,10,13\}$ & $\alpha_8$ & $\{10,13\}$\\
		$\{10,10,13\}$ & $\alpha_9$ & $\{10,13\}$ \\
		$\{10,11,13\}$ &$\alpha_{10}$ & $\{10,13\}$ \\
		$\{10,12,13\}$ & $\alpha_{12}$ & $\{10,13\}$ \\
		$\{10,13,13\}$ & $\alpha_{11}$ & $\{10,13\}$ \\
		$\{10,13,14\}$ & $\alpha_6$ & $\{10,13\}$ \\
		$\{10,13,15\}$ & $\alpha_1$ & $\{10,13\}$ \\
		$\{10,13,16\}$ & $\alpha_{17}$ & $\{13,16\}$ \\
		$\{11,13,16\}$ & $\alpha_{17}$ & $\{13,16\}$ \\
		$\{12,13,16\}$ & $\alpha_{17}$ & $\{13,16\}$ \\
		$\{13,13,16\}$ & $\alpha_5$ & $\{13,16\}$ \\
		$\{13,14,16\}$ & $\alpha_5$ & $\{13,16\}$ \\
		$\{13,15,16\}$ & $\alpha_2$ & $\{13,16\}$ \\
		$\{13,16,16\}$ & $\alpha_2$ & $\{13,16\}$ \\
		$\{13,16,17\}$ & $\alpha_5$ & $\{13,16\}$ \\
		$\{13,16,18\}$ & $\alpha_2$ & $\{13,16\}$ \\
		$\exp((G_{34},G(3,3,3)),\kappa)=\{13,16,19\}$\\
				\hline
	\end{tabular}
	\medskip
		\caption{Induction tables for $\CA=(G_{34},A_2),(\CA^H,\kappa)=((G_{34},A_3),\kappa),((G_{34},G(3,3,3)),\kappa)$}
	\label{tableG34A2}
\end{table}

(2). Secondly, we checked that in each instance in Table \ref{tableB}  when $\CA$ is of rank $5$, none of the $(\CA^H, \kappa)$ is inductively free.  Specifically, we proved a slightly stronger statement, namely that none of the multiarrangements $((G_{34},A_1^2), \kappa)$ and  $((G_{34},A_2), \kappa)$ is additively free, cf.~Definition \ref{def:order}.
This was accomplished by computer calculations carried out in SAGE \cite{sage} as follows: 
Let $\CA$ be $(G_{34},A_1^2)$ or $(G_{34},A_2)$.
If $(\CA,\kappa)$ were additively free with a chain of free multiplicities $(\CA,\mu_i)$ starting at the empty $4$-arrangement, ending at  $(\CA,\kappa)$, then according to Theorem \ref{thm:add-del},  for each Euler restriction with respect to $H_i$ say, in the chain $(\CA,\mu_i)$, we get 
$\exp(\CA^{H_i},\mu_i^*)\subset\exp(\CA,\mu_i)$ at every step.
    In order to show that $(\CA,\kappa)$ does not admit such a chain, we start with the free arrangement $(\CA,\kappa)$ along with its set of exponents and successively eliminate hyperplanes. Starting with the exponents of 
$(\CA,\kappa)$, we assign 
``virtual exponents'' to a deletion along such a chain recursively as follows.
We determine the number of hyperplanes in the corresponding Euler restriction $(\CA^{H_i},\mu_i^*)$ for any $H_i\in\CA$. Theorem \ref{thm:add-del} can only be employed if $\vert\mu_i^*\vert$ equals the sum of $3$ out of $4$ ``virtual exponents'' of  $(\CA,\mu_i)$. So if we do have such an equality, we save the corresponding deletion $(\CA',\mu_i') = (\CA,\mu_{i-1})$ 
in a list, lower the ``virtual exponent'' that was not part of the sum by $1$ and assign this new set of  ``virtual exponents'' in the next round of our algorithm to $(\CA,\mu_{i-1})$. We do stress 
the ``virtual exponents'' that are being determined along the way need not be actual exponents of $(\CA,\mu_{i-1})$, as the latter may not be free. Once every hyperplane $H_i\in\CA_i$ has been tested, we repeat the algorithm with 
each of the deletions $(\CA,\mu_{i-1})$ along with its set of ``virtual exponents''
that was incurred in the previous round.

It turns out that, 
as a result of this algorithm, 
every chain of subarrangements $(\CA,\mu_i)$ admitting a correct number of hyperplanes in each restriction such that Theorem \ref{thm:add-del} could in principal be applied does terminate well before reaching the empty arrangement. 
Consequently, there is no free chain of subarrangements for $(\CA,\kappa)$, so in particular,  
$(\CA,\kappa)$ fails to be 
additively free, so it is not 
inductively free.

%%%%%%%%%%%%%%%%%%%%%%%%%%%%%%%%%%%%%%%%%%%%%%%%%%%%%%%%%%%%%%%%%%%%%%
%%%%%%%%%%%%% Acknowledgments
%%%%%%%%%%%%%%%%%%%%%%%%%%%%%%%%%%%%%%%%%%%%%%%%%%%%%%%%%%%%%%%%%%%%%%
\bigskip {\bf Acknowledgments}:
This work was supported by DFG-Grant
RO 1072/21-1 (DFG Project number 494889912) to G.~R\"ohrle. 
We thank the anonymous referees for several suggestions which have improved the exposition.

%%%%%%%%%%%%%%%%%%%%%%%%%%%%%%%%%%%%%%%%%%%%%%%%%%%%%%%%%%%%%%%%%%%%%%
%%%%%%%%%%%%% bibliography
%%%%%%%%%%%%%%%%%%%%%%%%%%%%%%%%%%%%%%%%%%%%%%%%%%%%%%%%%%%%%%%%%%%%%%

%\bigskip

\bibliographystyle{amsalpha}

\begin{thebibliography}{ATW08}

	\bibitem[ATW08]{abeteraowakefield:euler}
	T.~Abe, H.~Terao, and M.~ Wakefield, 
	\emph{The Euler multiplicity and addition-deletion theorems for multiarrangements}. 
	J. Lond. Math. Soc. (2)  \textbf{77} (2008), no. 2, 335--348.
		
	\bibitem[AHR14]{amendhogeroehrle:indfree}
	N.~Amend, T.~Hoge and G.~R\"ohrle, 
	\emph{On inductively free restrictions of reflection arrangements},
	 J. Algebra \textbf{418} (2014), 197--212.
		
	\bibitem[HR13]{hogeroehrle:free} 
	T.~Hoge and G.~R\"ohrle, 
	\emph{Reflection arrangements are hereditarily free}, T\^ohoku Math.~J.
	~\textbf{65} (2013), no.~3, 313--319.
	
	\bibitem[HR15]{hogeroehrle:indfree} 
	\bysame, %T.~Hoge and G.~R\"ohrle, 
	\emph{On inductively free reflection arrangements}, J.~Reine Angew.~Math.
	\textbf{701} (2015), 205--220. 
		
	\bibitem[HR18]{hogeroehrle:Ziegler} 
	\bysame, %T.~Hoge and G.~R\"ohrle, 
	\emph{Inductive freeness of Ziegler's canonical multiderivations for reflection arrangements}. J. Algebra \textbf{512} (2018), 357--381. 

	\bibitem[HR25]{hogeroehrle:ZieglerII} 
    \bysame, %	T.~Hoge, G.~R\"ohrle, 
	\emph{Inductive Freeness of Ziegler’s Canonical Multiderivations}. Discrete Comput. Geom. (2025). \url{https://doi.org/10.1007/s00454-024-00644-y}
	
	\bibitem[HRS17]{hogeroehrleschauenburg:free} 
	T.~Hoge, G.~R\"ohrle and A.~Schauenburg, 
	\emph{Inductive and Recursive Freeness of Localizations of 
		multiarrangements}, 
	in: Algorithmic and Experimental Methods in Algebra, Geometry, and Number Theory, Springer Verlag
	2017.
		
	\bibitem[OS82]{orliksolomon:unitaryreflectiongroups}
	P.~Orlik and L.~Solomon,
	\emph{Arrangements Defined by Unitary Reflection Groups},
	Math. Ann. \textbf{261}, (1982), 339--357.
		
	\bibitem[OT92]{orlikterao:arrangements} 
	P.~Orlik and H.~Terao,
	\emph{Arrangements of hyperplanes}, Springer-Verlag, 1992.
	
	\bibitem[OT93]{orlikterao:free} 
	\bysame, % P.~Orlik and H.~Terao,
	\emph{Coxeter arrangements are 
		hereditarily free}, T{\^{o}}hoku Math.  J. \textbf{45} (1993), 369--383.
	
	\bibitem[ST54]{shephardtodd}
	G.C. Shephard and J.A. Todd, 
	\emph{Finite unitary reflection groups}.
	Canadian J. Math. \textbf{6}, (1954), 274--304. 

	\bibitem[St15]{sage}
W.~A.~Stein, et al,
\emph{Sage Mathematics Software (Version 7.5.1)}, 
The Sage Development Team, 2015, 
\url{http://www.sagemath.org}.

		
	\bibitem[S64]{steinberg:differential}
	R. Steinberg,
	\emph{Differential equations invariant under finite reflection groups}, 
	Trans. Amer. Math. Soc. \textbf{112}, (1964), 392--400.
	
	\bibitem[T80a]{terao:freeI} 
	H.~Terao, \emph{Arrangements of hyperplanes and
		their freeness I, II}, J. Fac. Sci.  Univ. Tokyo \textbf{27} (1980),
	293--320.
	
	\bibitem[T80b]{terao:freereflections} 
	\bysame, %H.~Terao, 
	\emph{Free arrangements of hyperplanes and unitary reflection groups}. 
	Proc. Japan Acad. Ser. A Math. Sci. \textbf{56} (1980), no.~8, 389--392. 
		
	\bibitem[Z89]{ziegler:multiarrangements}
	G.~Ziegler, 
	\emph{Multiarrangements of hyperplanes and their freeness}. 
	Singularities (Iowa City, IA, 1986), 345--359,
	Contemp. Math., \textbf{90}, Amer. Math. Soc., Providence, RI, 1989. 
\end{thebibliography}

\newcommand{\etalchar}[1]{$^{#1}$}
\providecommand{\bysame}{\leavevmode\hbox to3em{\hrulefill}\thinspace}
\providecommand{\MR}{\relax\ifhmode\unskip\space\fi MR }
\providecommand{\MRhref}[2]{%
	\href{http://www.ams.org/mathscinet-getitem?mr=#1}{#2} }
\providecommand{\href}[2]{#2}

%%%%%%%%%%%%%%%%%%%%%%%%%%%%%%%%%%%%%%%%%%%%%%%%%%%%%%%%%%%%%%%%%%%%%%
%%%%%%%%%%%%%%%%%%%%%%%%%%%%%%%%%%%%%%%%%%%%%%%%%%%%%%%%%%%%%%%%%%%%%%

\end{document}